\documentclass[12pt,a4paper,reqno]{amsart} 
\usepackage[titletoc]{appendix}
\usepackage[margin=1in]{geometry}
\usepackage{amsfonts,amssymb,amscd,amsthm,amsmath,mathtools,hyperref,fancybox,color,float,xcolor,mathrsfs,cleveref}
\usepackage[shortlabels]{enumitem}

\usepackage[noadjust,compress]{cite}
\mathchardef\ordinarycolon\mathcode`\:
\mathcode`\:=\string"8000
\begingroup \catcode`\:=\active
\gdef:{\mathrel{\mathop\ordinarycolon}}
\endgroup
\makeatletter
\renewcommand*\env@matrix[1][\arraystretch]{%
	\edef\arraystretch{#1}%
	\hskip -\arraycolsep
	\let\@ifnextchar\new@ifnextchar
	\array{*\c@MaxMatrixCols c}}
\makeatother
\theoremstyle{plain}
\newtheorem{theorem}{Theorem}[section]
\newtheorem{lemma}[theorem]{Lemma}
\newtheorem{corollary}[theorem]{Corollary}

\theoremstyle{definition}
\newtheorem{definition}[theorem]{Definition}
\newtheorem{example}[theorem]{Example}
\newtheorem{hypothesis}[theorem]{Hypothesis}

\theoremstyle{remark}
\newtheorem{remark}[theorem]{Remark}

\title[Stability of Reset and Impulsive Continuous-time Linear Switched Systems]{Stability of Reset and Impulsive Linear Switched Systems} 
\author[Swapnil Tripathi]{Swapnil Tripathi}
\address{Department of Mathematics\\
	Indian Institute of Science Education and Research Bhopal\\
	Bhopal Bypass Road, Bhauri \\
	Bhopal 462 066, Madhya Pradesh\\
	India}
\email{swap\_trip@outlook.com}
\author[Nikita Agarwal]{Nikita Agarwal}
\address{Department of Mathematics\\
	Indian Institute of Science Education and Research Bhopal\\
	Bhopal Bypass Road, Bhauri \\
	Bhopal 462 066, Madhya Pradesh\\
	India}
\email{nagarwal@iiserb.ac.in}

\date{\today}

\begin{document}
	
\begin{abstract}
	We study stability issue of reset and impulsive switched systems. We find time constraints (dwell time and flee time) on switching signals which stabilize a given reset switched system. For a given collection of matrices, we find an assignment of resets and time constraints on switching signals which guarantee stability of the reset switched system. Similar results are obtained for impulsive switched systems as well. Two techniques, namely, analysis of flow of the system and the multiple Lyapunov function approach is used to obtain the results. The results are later generalized to obtain mode-dependent time constraints for stability of these systems.
\end{abstract}

\maketitle



\section{Introduction}

A continuous-time switched linear system is a special kind of time-variant system, which comprises a family of linear time-invariant subsystems and a switching law which determines the active subsystem at any given instant. For switched systems, various stability issues, such as, stability under arbitrary switching~\cite{agrachev2001lie,narendra1994common,yang2012sufficient} and stability under constrained switching~\cite{agarwal2018simple,hespanha1999stability,tripathi2021bistable,tripathi2022unistable}, are addressed in the literature. These issues stem from the fact that a switched system can be unstable even if all the subsystems involved are stable~\cite{liberzon2003switching}. When working with switched systems, it is usually assumed that the state evolves continuously at the switching instances. However, impulses are commonly experienced in real life systems when there are abrupt changes in the system, for instance, in a router with multiple buffers~\cite{li2005switched} and certain drug administration procedures~\cite{lakshmikantham1989theory}. It is, thus, natural to consider switched systems with jumps or impulses at switching instances. 

Impulses experienced by a system are usually viewed as either disturbances, for example, packet loss during data transmission through a router, or a consequence of impulsive control of continuous-time systems~\cite{chen2014designing}. A well-known approach to achieve stability is resetting the state at switching instances, see~\cite{paxman2003stability} for switched systems and~\cite{xiao2019stabilization} for switched singular systems. Here, the state jump occurring at any given switching instant depends on the subsystems active before and after the switching instant. Such systems are called \textit{reset switched systems}. It is known that when all subsystems are asymptotically stable, there exist resets such that the reset switched system is asymptotically stable~\cite{bras2012stability}. However, if the set of available resets is restricted, it may not always be possible to stabilize the reset switched system no matter how the resets are assigned.

Reset switched systems have been extensively studied in the literature. These were first introduced in~\cite{hespanha2002switching} where the authors propose to use resets in the state of a multi-controller in order to stabilize a given plant. Similar architecture is studied in~\cite{hespanha2007optimal,paxman2003stability}. Reset switched systems model processes where all components of the state are available for reset. When not all components of the state are available for reset, stability of specific systems can also be achieved using partial resets~\cite{bras2012stability}. It is known that systems consisting of partially commuting stable subsystems can be stabilized using partial state resets, see~\cite{bras2017stability}. Also in~\cite{carapito2019stabilization}, it is proved that a positive switched system can always be stabilized by using partial state resets from the set of Metzler matrices. Unlike reset switched systems where the state reset depends on the subsystems active before and after the switching instant, there is a class of impulsive switched systems in which the state jumps depend just on the switching instant. These systems are, thus, fundamental in understanding real-life systems with external disturbances. Several issues such as finite-time stability~\cite{wang2013finite}, and asymptotic stabilization of impulsive switched systems using feedback control~\cite{xu2008lmi} have been addressed in the literature. 

In this paper, we will study the dynamics of both reset switched systems and impulsive switched systems. As mentioned earlier, asymptotic stabilization of switched systems using state resets and constrained switching have extensively, and independently, been explored in the literature. However, to the best of our knowledge, these issues have not been discussed simultaneously. In this paper, we will explore this issue. Two main questions arise in this study. The first one is concerned with finding time constraints (dwell time and flee time) on switching signals which stabilize a given reset switched system. The other one pertains to the case when resets are not preassigned but a large collection is given from which resets can be chosen. In this case, the natural question is of finding an assignment of resets and time constraints on switching signals which guarantee stability of the reset switched system. Similar questions are asked for impulsive switched systems as well. We also discuss the situation when the set of resets or impulses, as the case may be, is a convex hull of matrices. Further we obtain mode-dependent dwell time and flee time constraints for stability of the reset and impulsive switched systems.

\section{Preliminaries and Problem Setup}

In this section, we first define the three types of time-dependent continuous-time linear switched systems. Then we define various stability notions for these systems. We then discuss notions of dwell time and flee time for a switching signal. At the end of this section, we will discuss issues that will be the focus of this work. 

\subsection{A Linear Switched System}
Let $\{A_p\colon\,p\in\mathcal{P}\}$ be a finite collection of real matrices of size $n$, where $\mathcal{P}$ is a fixed finite index set. A \textit{time-dependent linear switched system}, $\Sigma_\sigma$, is defined as
\begin{eqnarray}\label{eq:system}
	\Sigma_\sigma\colon\ \dot{x}(t)&=&A_{\sigma(t)}x(t), 
\end{eqnarray}	
where $x(t)\in\mathbb{R}^n$ and $\sigma\colon[0,\infty)\to \mathcal{P}$ is a right-continuous piecewise constant function, known as a \textit{switching signal}. Here $\dot{x}(t)$ denotes the derivative of the state $x(t)$ with respect to the time variable $t$. The switched system~\eqref{eq:system} has $\lvert\mathcal{P}\rvert$ many subsystems given by $\dot{x}(t)=A_px(t)$ for each $p\in \mathcal{P}$. If $\sigma(t)=p$, the dynamics of the switched system is governed by the subsystem $\dot{x}(t)=A_{p}x(t)$. The \textit{switching signal} $\sigma$ is, thus, a rule determining the \textit{active subsystem} at any time instant $t$. 

\subsection{The Classic Reset Switched System}
Let $\{A_p\colon\,p\in\mathcal{P}\}$ be a finite collection of real matrices of size $n$, where $\mathcal{P}$ is a fixed finite index set, as given above. Let $\sigma\colon[0,\infty)\to \mathcal{P}$ be a switching signal (with properties defined above) having the set of discontinuities as $\{t_k\}_{k\ge 1}$ satisfying $t_i<t_{i+1}$, for all $i\ge 1$. Let $\mathcal{R}=\{R_{(p,q)}\colon\ p,q\in\mathcal{P}\}$ be a collection of real matrices of size $n$. The matrices in this collection are known as \textit{reset matrices}. A \textit{reset switched system}, $\Sigma_\sigma^{\mathcal{R}}$, is defined as \begin{eqnarray}\label{eq:RSS}
	\Sigma_\sigma^{\mathcal{R}}\colon\ \ 
		\dot{x}(t)=A_{\sigma(t)}x(t), \ \ x(t_k)=R_{(\sigma(t_{k-1}),\sigma(t_k))}\, x(t_k^-),
\end{eqnarray}
where $x(t_k^-)=\lim_{t\to t_k^-}x(t)$. Thus, in a reset switched system, at each switching instant, the state $x(t)$ changes according to a certain linear reset map from the collection $\mathcal{R}$. The reset map does not depend on the time at the switching instant, but on the subsystems active just before and just after the switching instant. In particular, if the subsystem corresponding to the matrix $A_p$ is active just before the switching time $t_k$ and the subsystem corresponding to the matrix $A_q$ is active just after the switching time $t_k$, then $R_{(\sigma(t_{k-1}),\sigma(t_k))}=R_{(p,q)}$. This gives rise to a switched system with linear jumps at the switching instances as long as not all reset matrices $R_{(\sigma(t_{k-1}),\sigma(t_k))}$, for $k\ge 1$, are identity matrices. Such a system has a discontinuous flow at several (possibly all) switching instants. 
An example is shown in \Cref{fig:ResetSwitchedSystem}.

\begin{figure}[h!]
	\centering
	\includegraphics[width=\textwidth]{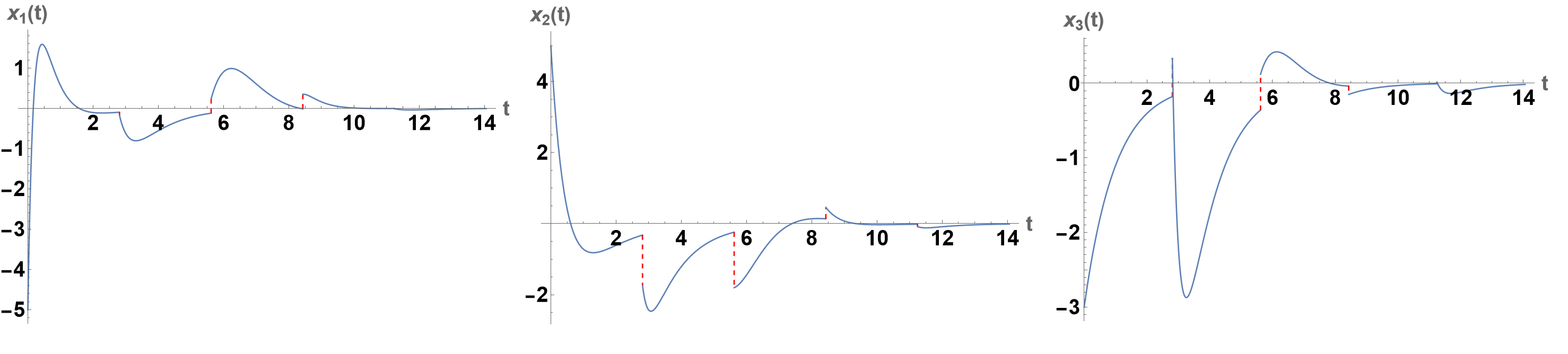}
	\caption{Refer to \Cref{eg:3dArbReset}. The reset switched system with subsystems $A_8$, $A_9$, $A_{10}$ and $R_{(p,q)}=M_2$, for all $p,q\in \{1,2,3\}$. Here, $(x_1(t),x_2(t),x_3(t))^\top$ is the flow of the reset switched system with initial condition $(-5,5,-3)^\top$ under the periodic switching signal $\sigma$ satisfying $t_{k-1}-t_k=2.81$ and $\sigma(t_k)=(k+1)\pmod 3$, for all $k\in\mathbb{N}$. Dashed trajectory denotes jumps at switching instances.}
	\label{fig:ResetSwitchedSystem}
\end{figure}

\subsection{A Linear Impulsive Switched System}
Let $\{A_p\colon\,p\in\mathcal{P}\}$ be a finite collection of real matrices of size $n$, where $\mathcal{P}$ is a fixed finite index set, as given above. Let $\sigma\colon[0,\infty)\to \mathcal{P}$ be a switching signal (with properties defined above) having the set of discontinuities $\{t_k\}_{k\ge 1}$ satisfying $t_k<t_{k+1}$, for all $k\ge 1$. Consider the collection $\mathcal{I}=\{I_q\colon\,q\in\mathcal{Q}\}$ of real matrices of size $n$, where $\mathcal{Q}$ is a fixed index set. Each matrix in this collection is known as a \textit{linear impulse}. Let $\omega\colon\{t_1,t_2,\dots\}\to\mathcal{Q}$ be a function. A \textit{linear impulsive switched system}, $\Sigma_{\sigma,\omega}^{\mathcal{I}}$, is defined as  \begin{eqnarray}\label{eq:ISS}
	\Sigma_{\sigma,\omega}^{\mathcal{I}}\colon \ \ 
		\dot{x}(t)=A_{\sigma(t)}x(t), \ \ x(t_k)=I_{\omega(t_k)}\, x(t_k^-),
\end{eqnarray}
where $x(t_k^-)=\lim_{t\to t_k^-}x(t)$. 

Observe that, by setting $\mathcal{Q}=\mathcal{P}\times \mathcal{P}$ and $\omega(t_k)=(\sigma(t_k^-),\sigma(t_k))$ for each $k\ge 1$, the reset switched system~\eqref{eq:RSS} is indeed an impulsive switched system. Moreover, the impulsive switched system is more robust to disturbances than the reset switched system since it allows for resets at each switching instant that do not depend on the subsystem active just before and just after the switching instant. It, hence, serves as a model for systems where disturbances or impulses are dependent only on the time at the switching instant. It should also be noted that the presence of impulses may destabilize an otherwise stable system, examples will be discussed later.

\subsection{Stability notions}

Each of the switched systems~\eqref{eq:system},~\eqref{eq:RSS},~\eqref{eq:ISS} induces a flow on the state space $\mathbb{R}^n$. The dynamics of points on the state space under this flow gives rise to various stability notions which we define below. Let $x(t)$ denote the flow of a point $x(0)$, where the dependence on $\sigma$ has been suppressed for notational simplicity. The switched system is said to be  \begin{enumerate}
	\item \textit{stable} if for every $\epsilon>0$, there exists $\delta>0$ such that $\| x(0)\|<\delta$ implies $\|x(t)\|<\epsilon$ for all $t>0$, and
	\item \textit{asymptotically stable} if it is stable and $\Vert x(t)\Vert\to 0$ as $t\to\infty$ for every initial vector $x(0)\in\mathbb{R}^n$, 
\end{enumerate} 
The above definitions are valid for more general switched systems and not just for systems of the forms~\eqref{eq:system},~\eqref{eq:RSS}, and~\eqref{eq:ISS}. Note that the asymptotic stability is a strictly stronger notion than just stability, by definition. We say that the impulsive switched system~\eqref{eq:ISS}, for a given switching signal $\sigma$, is \textit{stable under arbitrary impulses} if the switched system~\eqref{eq:ISS} is stable for any choice of $\omega$. A matrix $A$ is said to be \textit{Hurwitz stable} if all its eigenvalues lie to the left of the imaginary axis. In this case, we will call $A$, just a \textit{stable matrix} throughout this paper. A matrix $A$ is known as an \textit{unstable matrix} if at least one of its eigenvalues lies to the right of the imaginary axis. With reference to the switched systems defined above, a subsystem $\dot{x}(t)=A_px(t)$ is said to be a \textit{stable subsystem} if the subsystem matrix $A_p$ is stable. It is known as an \textit{unstable subsystem} if the subsystem matrix is unstable. 

			
			\subsubsection{Switching Signals}\label{sec:NotationsAndIssues}
			Let $\{t_k\colon\,k\in\mathbb{N}\}$ be the set of discontinuities of a switching signal $\sigma$, which are known as the \textit{switching times} for the signal $\sigma$. We will consider only those switching signals $\sigma$ which satisfy the following property: $\{t_k\}$ is an increasing sequence with $t_k\uparrow \infty$ and $t_0=0$.
			
			\begin{definition}[Time constraints on signals -- Dwell time and Flee time]
				We say that a switching signal $\sigma$ has \textit{dwell time $\tau>0$ in a stable subsystem} if $t_{i+1}-t_i\ge \tau$ whenever $A_{\sigma(t_i)}$ is a stable matrix. We say that a signal $\sigma$ has \textit{flee time $\eta>0$ in an unstable subsystem} if $t_{i+1}-t_i\le \eta$ whenever $A_{\sigma(t_i)}$ is an unstable matrix. When all the subsystem matrices are stable, we only talk about dwell time. Similarly when all the subsystem matrices are unstable, we only talk about flee time.
			\end{definition}
		
			We now define the following classes of switching signals that will be used and referred to throughout the paper. For $\tau,\eta>0$, define:
			\begin{eqnarray*}\label{eq:class_signals}
				\mathcal{S}_s[\tau]&=&\left\{\sigma\colon[0,\infty)\to\mathcal{P}\vert\ t_k-t_{k-1}\ge \tau, \text{ for all }k\in\mathbb{N}\right\},\nonumber\\
				\mathcal{S}_u[\eta]&=&\left\{\sigma\colon[0,\infty)\to\mathcal{P}\vert\ t_k-t_{k-1}\le \eta, \text{ for all }k\in\mathbb{N}\right\},\nonumber\\
				\mathcal{S}[\tau,\eta]&=&\left\{\sigma\colon[0,\infty)\to\mathcal{P}\vert\ t_k-t_{k-1}\ge \tau,\ \text{ for stable } A_{\sigma(t_{k-1})},\text{ and }\right.\nonumber\\
				& &\hspace{85pt}\left.t_k-t_{k-1}\le \eta,\ \text{ for unstable } A_{\sigma(t_{k-1})}, \ k\in\mathbb{N}\right\}.
			\end{eqnarray*}
			
			The class $\mathcal{S}_s[\tau]$ is relevant only when all the subsystem matrices of the switched system are stable. As defined earlier, the quantity $\tau>0$ is known as dwell time for any signal $\sigma\in \mathcal{S}_s[\tau]$. Moreover the class $\mathcal{S}_u[\eta]$ is relevant only when all the subsystem matrices of the switched system are unstable. As defined earlier, the quantity $\eta>0$ is known as flee time for any signal $\sigma\in \mathcal{S}_u[\eta]$. Note that if all the subsystem matrices are stable, then the collection $\mathcal{S}[\tau,\eta]$ is nothing but the collection $\mathcal{S}_s[\tau]$. Further if all the subsystem matrices are unstable, then the collection $\mathcal{S}[\tau,\eta]$ is nothing but the collection $\mathcal{S}_u[\eta]$.
			
			As mentioned earlier, if the set of available impulses is restricted, it is not always possible to stabilize the switched system~\eqref{eq:ISS} under arbitrary choice of impulses. Simple examples exhibiting this can be constructed when the only available reset is the identity matrix. This leads us to the question of stabilizing a given reset switched system (or impulsive switched system) using time constraints on the signals. Moreover, if the set of resets (or impulses) is known, the question of finding time constraints for stability of switched system for arbitrary assignment of resets (or impulses) is an important issue and can be used to stabilize real-life systems with bounded disturbances. It is mathematically formulated as issue~\ref{issue4} below.
			
			The above discussion gives rise to the following issues pertaining to reset switched system~\eqref{eq:RSS} and impulsive switched system~\eqref{eq:ISS}.
			\begin{enumerate}[label={I.\arabic*}]
				\item\label{issue1} Given time constraints on $\sigma$, is there a collection of reset maps $\mathcal{R}$ so that the reset switched system $\Sigma_{\sigma}^{\mathcal{R}}$ is stable.
				\item\label{issue2} Given time constraints on $\sigma$, is there a collection of impulses $\mathcal{I}$ so that the impulsive switched system $\Sigma_{\sigma,\omega}^{\mathcal{I}}$ is stable for any $\omega$.
				\item\label{issue3} Given a set of resets $\mathcal{R}$, finding time constraints on $\sigma$ so that the reset switched system $\Sigma_{\sigma}^\mathcal{R}$ is stable.
				\item\label{issue4} Given a set of impulses $\mathcal{I}$, finding time constraints on $\sigma$ so that the impulsive switched system $\Sigma_{\sigma,\omega}^\mathcal{I}$ is stable for any $\omega$.
			\end{enumerate}

			In this paper, we will make use of two techniques to study the issues raised above dealing with stability for the switched systems. The first one is by analyzing the Euclidean norm of the flow of switched system, discussed in \Cref{sec:flow}. The other one is by constructing multiple Lyapunov functions, discussed in \Cref{sec:MLF}. We then present, in \Cref{sec:convexhull}, results for the case when the impulses experienced by the system lie in the convex hull of finitely many matrices. In \Cref{sec:MDDT}, we obtain mode dependent constraints for stability of impulsive switched systems, followed by numerical examples in \Cref{sec:examples}. 
			
			Throughout the paper, the collection $\mathfrak{s}\subseteq \mathcal{P}$ will consist of the indices of all stable subsystem matrices and $\mathfrak{u}\subseteq \mathcal{P}$ will consist of the indices of all unstable subsystem matrices. The set $M_n(\mathbb{R})$ denotes the set of $n\times n$ real matrices, and $\bf{O}$ denotes the zero matrix of size $n$. For a matrix $Q$ having all real eigenvalues, $\lambda_{\max}(Q)$ and $\lambda_{\min}(Q)$ denote the largest and the smallest eigenvalue of $Q$. The symbol $\|v\|$ denotes the Euclidean norm if $v\in\mathbb{R}^n$, and the spectral norm if $v\in M_n(\mathbb{R})$. If $J$ is a \textit{Jordan form} of a matrix $A$, then any matrix $P$ satisfying $A=PJP^{-1}$ (Jordan decomposition of $A$) will be referred to as a \textit{Jordan basis matrix} of $A$. We consider complex Jordan decomposition of matrices.
			
			\section{Stability results using the expression of the flow}\label{sec:flow}
			
			In this section, we consider the expression of the flow of a reset or an impulsive switched system and discuss the stability issues raised in \Cref{sec:NotationsAndIssues}. We derive inspiration from the techniques used in~\cite{agarwal2019stabilizing,karabacak2009dwell} to obtain the results. The bounds given in the following remark will be used throughout the paper.
			
			\begin{remark}\label{rem:linearsystemdecay}
				For each $p\in \mathcal{P}$, let $J_p$ be the complex Jordan form of the subsystem matrix $A_p$, which is unique up to permutation of Jordan blocks. Then
				\begin{enumerate}
					\item if $A_p$ is unstable, then there exist $c_p,\mu_p^*>0$ satisfying $\|{\rm{e}}^{J_p t}\|\le c_p\,{\rm{e}}^{\mu_p^* t}$, for all $t>0$. Clearly $\mu_p^*\ge\mu_p$, where $\mu_p>0$ is the real part of the eigenvalue of $A_p$ farthest to the imaginary axis on its right. 
					\item If $A_p$ is stable, then there exist $c_p,\lambda_p^*>0$ satisfying $\|{\rm{e}}^{J_p t}\|\le c_p\,{\rm{e}}^{-\lambda_p^* t}$, for all $t>0$. Clearly $-\lambda_p^*\ge-\lambda_p$, where $-\lambda_p<0$ is the real part of the eigenvalue of $A_p$ closest to the imaginary axis on its left. 
				\end{enumerate}
			\end{remark} 
			
			For each $p\in\mathcal{P}$, fix a complex Jordan decomposition $P_p J_pP_p^{-1}$ of $A_p$. The flow of the reset switched system~\eqref{eq:RSS} is given by \begin{eqnarray*}
				x(t)&=&{\rm{e}}^{A_{\sigma(t_j)}(t-t_j)}\left(\prod_{k=1}^{j}R_{(\sigma(t_{k-1}),\sigma(t_{k}))}\,{\rm{e}}^{A_{\sigma(t_{k-1})}(t_{k}-t_{k-1})}\right)x(0),
			\end{eqnarray*}
			for $t\in[t_j,t_{j+1})$, $j\ge 1$, which can be re-written as
			\begin{multline}\label{eq:flow:RSS-new}
				x(t)=P_{\sigma(t_j)}{\rm{e}}^{J_{\sigma(t_j)}(t-t_j)}\left(\prod_{k=1}^{j}P_{\sigma(t_{k})}^{-1}R_{(\sigma(t_{k-1}),\sigma(t_{k}))}P_{\sigma(t_{k-1})}{\rm{e}}^{J_{\sigma(t_{k-1})}(t_{k}-t_{k-1})}\right)\\P_{\sigma(t_{0})}^{-1}x(0).
			\end{multline}
			Similarly, the flow of the impulsive switched system~\eqref{eq:ISS} is given by \begin{eqnarray*}
				x(t)&=&{\rm{e}}^{A_{\sigma(t_j)}(t-t_j)}\left(\prod_{k=1}^{j}I_{\omega(t_k)}\,{\rm{e}}^{A_{\sigma(t_{k-1})}(t_{k}-t_{k-1})}\right)x(0),
			\end{eqnarray*}
			for $t\in[t_j,t_{j+1})$, $j\ge 1$, which can be re-written as
			\begin{eqnarray*}
				x(t)&=&P_{\sigma(t_j)}{\rm{e}}^{J_{\sigma(t_j)}(t-t_j)}\left(\prod_{k=1}^{j}P_{\sigma(t_{k})}^{-1}I_{\omega(t_k)}P_{\sigma(t_{k-1})}{\rm{e}}^{J_{\sigma(t_{k-1})}(t_{k}-t_{k-1})}\right)P_{\sigma(t_{0})}^{-1}x(0).
			\end{eqnarray*}
			
			It is clear that if the reset (impulse) at a given switching instant is the zero matrix $\bf{O}$, then the trajectory jumps to the origin and stays thereafter. Hence to obtain non-trivial results concerning stability, we assume that all the resets (impulses) are nonzero matrices. We will refer to these as \textit{nonzero resets} (\textit{nonzero impulses}).
			
			In the following set of results, we prove that for any signal satisfying a given dwell time constraint in stable subsystems and a given flee time constraint in unstable subsystems, it is possible to obtain resets (impulses) so as to ensure stability of the associated reset (impulsive) switched system. 
			
			\begin{theorem}\label{thm:ResetsGivenConstraints}
				Given a switched system~\eqref{eq:system} with $\sigma\in \mathcal{S}[\tau,\eta]$, there exists a collection $\mathcal{R}$ of nonzero resets such that the corresponding reset switched system~\eqref{eq:RSS} is stable.\end{theorem}
			\begin{proof}
				Suppose $\mathcal{R}=\{R_{p,q}\ : \ p,q\in \mathcal{P}\}$ be a collection of nonzero resets. We will find conditions on this collection which ensure stability of the corresponding reset switched system. For each $p,q\in\mathcal{P}$, let $A_p=P_pJ_pP_p^{-1}$ be a complex Jordan decomposition of $A_p$ and let $K_{p,q}(t)=\|P_q^{-1}R_{(p,q)}P_p{\rm{e}}^{J_p t}\|$. The term $K_{p,q}(t_k-t_{k-1})$ is precisely the one that appears in the product in the parentheses in~\eqref{eq:flow:RSS-new} with $\sigma(t_{k-1})=p$ and $\sigma(t_k)=q$. Consider a signal $\sigma\in \mathcal{S}[\tau,\eta]$. Using the expression~\eqref{eq:flow:RSS-new} of the flow, we see that the reset switched system is stable if for each $p,q\in\mathcal{P}$, we have $K_{p,q}(t)\le 1$, for all $0<t\le \eta$ when $p\in\mathfrak{u}$, and $t\ge \tau$ when $p\in\mathfrak{s}$. Using \Cref{rem:linearsystemdecay},  
				\[
				K_{p,q}(t) \le \|P_q^{-1}R_{(p,q)}P_p\| \|{\rm{e}}^{J_p t}\| \le \begin{cases}
					\|P_q^{-1}R_{(p,q)}P_p\| c_p\,{\rm{e}}^{\mu_p^* t}, & p\in\mathfrak{u},\\	\|P_q^{-1}R_{(p,q)}P_p\| c_p\,{\rm{e}}^{-\lambda_p^* t}, & p\in\mathfrak{s}.
				\end{cases}
				\]
				Hence, $K_{p,q}(t)\le 1$ if $\|P_q^{-1}R_{(p,q)}P_p\|\le c_p^{-1}{\rm{e}}^{-\mu_p^*\eta}$ when $p\in\mathfrak{u}$, and if $\|P_q^{-1}R_{(p,q)}P_p\|\le c_p^{-1}{\rm{e}}^{\lambda_p^*\tau}$ when $p\in\mathfrak{s}$. Thus consider the convex set of matrices given by 
				\begin{eqnarray}\label{eq:US}
					U_{p,q}^\eta&=&\{M\in M_n(\mathbb{R})\colon\ \|P_q^{-1}MP_p\|\le c_p^{-1}{\rm{e}}^{-\mu_p^*\eta}\}\setminus\{\bf{O}\},\nonumber\\
					S_{p,q}^\tau&=&\{M\in M_n(\mathbb{R})\colon\ \|P_q^{-1}MP_p\|\le c_p^{-1}{\rm{e}}^{\lambda_p^*\tau}\}\setminus\{\bf{O}\}.
				\end{eqnarray} 
				Then the reset switched system is stable for the collection $\mathcal{R}$ of resets, where $R_{(p,q)}\in U_{p,q}^\eta$ if $p\in\mathfrak{u}$ and $R_{(p,q)}\in S_{p,q}^\tau$ if $p\in\mathfrak{s}$. 
			\end{proof}
			
			\begin{remark}
				In the literature, the reset matrices are assumed to be nonsingular.  The above procedure also allows us to choose nonsingular matrices as the choice of stabilizing resets since nonsingular matrices are dense in the set of all matrices, and hence intersect with the sets $U_{p,q}^\eta$ and $S_{p,q}^\tau$. The convex sets $U_{p,q}^\eta$ and $S_{p,q}^\tau$ also contain nonzero singular matrices as well. The existence of such matrices can be established by showing the existence of pairs of linearly independent matrices $(A,B)$ satisfying $\det(AB)<0$.
			\end{remark}
			
			The preceding theorem yields an alternate proof for the fact that a switched system with all subsystems stable can be stabilized under arbitrary switching by choosing suitable reset matrices. This result is proved in~\cite[Theorem 2]{bras2012stability} using multiple Lyapunov functions approach.
			
			\begin{corollary}\label{cor:StabilizingResets}
				For a switched system~\eqref{eq:system} having all stable subsystems, there exists a collection $\mathcal{R}$ of nonzero resets so that the corresponding reset switched system~\eqref{eq:RSS} is stable under arbitrary switching.
			\end{corollary}
			
			\begin{proof}
				Let $d_p>c_p$ and $R_{(p,q)}=d_p^{-1}P_qP_p^{-1}$, for each $p,q\in\mathcal{P}$. Then for all $t\ge 0$, $\|P_q^{-1}R_{(p,q)}P_p{\rm{e}}^{J_p t}\|=d_p^{-1}\Vert e^{J_pt}\Vert \le d_p^{-1}c_pe^{-\lambda_p^*t}<1$. Hence the result follows.
			\end{proof}

			\begin{theorem}
				For a given switched system~\eqref{eq:system} with $\sigma\in \mathcal{S}[\tau,\eta]$, there exists a collection $\mathcal{I}$ of non-zero impulses such that the corresponding impulsive switched system~\eqref{eq:ISS} is stable for each $\omega$.
			\end{theorem}
			
			\begin{proof}
				Define $\mathcal{I} = \left(\bigcap_{p\in\mathfrak{s}, q\in\mathcal{P}}S^\tau_{p,q}\right)\cap\left(\bigcap_{p\in\mathfrak{u}, q\in\mathcal{P}} U^\eta_{p,q}\right)$. Then $\|P_q^{-1}MP_p\|\le C(\tau,\eta)$, for all $M\in\mathcal{I}$, $p,q\in\mathcal{P}$, where\\ $C(\tau,\eta)=\min\left\{\min_{p\in \mathfrak{u}} \,\, c_p^{-1}{\rm{e}}^{-\mu_p^*\eta},\,\min_{p\in \mathfrak{s}}\,\, c_p^{-1}{\rm{e}}^{\lambda_p^*\tau}\right\}$. Note that, for fixed dwell time $\tau$, as flee time $\eta$ decreases, $C(\tau,\eta)$ is non-decreasing. Hence the collection $\mathcal{I}$ expands. Further, for fixed flee time $\eta$, as dwell time $\tau$ increases, $C(\tau,\eta)$ non-decreasing. Hence the collection $\mathcal{I}$ expands. 
			\end{proof}
			
			Now we focus on the converse problem of computing dwell time constraint $\tau$ and flee time constraint $\eta$ for which the given reset switched system~\eqref{eq:RSS} is stable for all signals $\sigma\in\mathcal{S}[\tau,\eta]$. Note that if the dwell time $\tau$ is sufficiently large, then $\|P_q^{-1}R_{(p,q)}P_p{\rm{e}}^{J_p t}\|<1$, for all $p\in\mathfrak{s}$. However, choosing a suitable flee time $\eta$ is not so straightforward. To this end, we will work with switched systems where switching from an unstable susbsystem is restricted. We now describe the notion of graph dependent switched systems where the underlying graph has subsystems as vertices and the edges between vertices determine the allowed switchings between associated subsystems. Stability issues for such systems (without resets/impulses) are well-studied in the literature, see~\cite{agarwal2018simple,agarwal2019stabilizing,karabacak2013dwell}. 
			
			\subsection*{Graphs}
			
			A \textit{directed graph} is an ordered pair $\mathcal{G}=(\mathcal{V}(\mathcal{G}),\mathcal{E}(\mathcal{G}))$, where $\mathcal{V}(\mathcal{G})$ is the set of \textit{vertices} and $\mathcal{E}(\mathcal{G})$ is a set of ordered pair of vertices, called \textit{directed edges}. If $(u,v)\in\mathcal{E}(\mathcal{G})$ for some $u,v\in\mathcal{V}(\mathcal{G})$, we say that there is a directed edge from vertex $u$ to vertex $v$. A \textit{loop} is a directed edge that connects a vertex to itself. A graph is said to be \textit{simple} if there are no loops or multiple directed edges from one vertex to another. A \textit{cycle} is a sequence of directed edges $(u_1,u_2)$, $(u_2,u_3)$, $\dots$, $(u_k,u_{k+1})$, with $u_{k+1}=u_1$. A graph without any cycles is called an \textit{acyclic graph}. A \textit{topological ordering} of a directed graph is a total ordering of its vertices such that for every edge, the starting vertex of the edge occurs before the ending vertex of the edge. It is known that a directed graph has a topological ordering if and only if it is acyclic.
			
			\subsection*{Switching governed by an underlying graph}
			
			Recall that $\{t_k\}_{k\in\mathbb{N}}$ denotes the increasing sequence of switching times of a switching signal $\sigma$ with $t_0=0$. We now define switching governed by an underlying graph as follows.
			
			\begin{definition}[Switching governed by an underlying graph]
				Let $\{A_p\colon p\in\mathcal{P}\}$ be a finite set of subsystem matrices, each of size $n$. Let $\mathcal{G}$ be a directed graph with vertices $\mathcal{V}(\mathcal{G})=\{A_p\colon p\in\mathcal{P}\}$ and a set of directed edges, $\mathcal{E}(\mathcal{G})$. Assume that for each $p\in\mathcal{P}$, there exists a $q\in\mathcal{P}$ such that $(p,q)\in \mathcal{E}(\mathcal{G})$. Any signal satisfying $(A_{\sigma(t_{k-1})},A_{\sigma(t_{k})})\in\mathcal{E}(\mathcal{G})$, for all $k\in\mathbb{N}$ is called a \textit{$\mathcal{G}$-admissible} signal. We will denote the set of all $\mathcal{G}$-admissible signals by $\mathcal{S}^{\mathcal{G}}$. The switched system $\dot{x}=A_\sigma x$ is said to have \textit{switching governed by the graph $\mathcal{G}$} if the signal $\sigma$ is $\mathcal{G}$-admissible.
			\end{definition}
			
			For $\tau> 0$ and $\eta>0$, we define the class of $\mathcal{G}$-admissible signals with dwell time $\tau$ in stable and flee time $\eta$ in unstable subsystems as follows: $\mathcal{S}^{\mathcal{G}}[\tau,\eta]=\mathcal{S}^{\mathcal{G}}\cap\mathcal{S}[\tau,\eta]$. Hence, when all subsystems are stable, we define $\mathcal{S}^{\mathcal{G}}_s[\tau]=\mathcal{S}^{\mathcal{G}}\cap\mathcal{S}_s[\tau]$. Similarly, when all subsystems are unstable, we define $\mathcal{S}^{\mathcal{G}}_u[\eta]=\mathcal{S}^{\mathcal{G}}\cap\mathcal{S}_u[\eta]$.

			%
				\noindent Consider the subgraphs $\mathcal{G}_s$ and $\mathcal{G}_u$ of $\mathcal{G}$ for which $\mathcal{V}(\mathcal{G}_s)=\mathcal{V}(\mathcal{G})=\mathcal{V}(\mathcal{G}_u)$ and \[
				\mathcal{E}(\mathcal{G}_s)=\{(A_p,A_q)\in\mathcal{E}(\mathcal{G})\colon\ p\in\mathfrak{s}\},\ \  \mathcal{E}(\mathcal{G}_u)=\{(A_p,A_q)\in\mathcal{E}(\mathcal{G})\colon\ p\in\mathfrak{u}\}.
				\] 
				That is, $\mathcal{E}(\mathcal{G}_s)$ consists of those edges of the graph $\mathcal{G}$ which originate from a stable subsystem and $\mathcal{E}(\mathcal{G}_u)$ consists of those edges of the graph $\mathcal{G}$ which originate from an unstable subsystem.
			
			\noindent Henceforth we will use the shorthand $(p,q)$ for the directed edge $(A_p,A_q)\in\mathcal{E}(\mathcal{G})$.
			
			\begin{lemma}[Adapted from~\cite{agarwal2018simple}]\label{lemma:acyclic_cond}
				(With the notations as above) For each $(p,q)\in\mathcal{E}(\mathcal{G})$, let $\mathbb{M}$ be a compact collection of matrices from $M_n(\mathbb{R})$ and let $\{c_p\colon\ p\in\mathcal{P}\}$ be a set of positive constants. If the subgraph $\mathcal{G}_u$ of $\mathcal{G}$ is acyclic, then for given $\epsilon\in (0,1)$, for each $p\in\mathcal{P}$, there is a Jordan basis matrix $P_p$ (depending on $\epsilon$) of $A_p$ so that \begin{eqnarray}\label{eq:acyclic_cond}
					\max\left\{c_p\left\|P_q^{-1}MP_p\right\|\colon (p,q)\in\mathcal{E}(\mathcal{G}_u), M\in\mathbb{M}\right\}<\epsilon.
				\end{eqnarray}
			\end{lemma}
			
			\begin{proof}
				Let $\epsilon\in (0,1)$ be given. For each $p\in\mathcal{P}$, let $Q_p$ be a Jordan basis matrix of $A_p$. Set $\rho=\max\left\{c_p\left\|{Q}_q^{-1}M{Q}_p\right\|\colon (p,q)\in\mathcal{E}(\mathcal{G}_u),M\in\mathbb{M}\right\}$. Since $\mathbb{M}$ is compact, $\rho$ is finite. If $\rho<\epsilon$, we can choose $P_p=Q_p$, for each $p$. Otherwise if $\rho\ge \epsilon$, we choose appropriate scaling factors to obtain Jordan basis matrix $P_p$ of $A_p$ such that~\eqref{eq:acyclic_cond} is satisfied. 
				
				Let $0<\xi<1$ and $\theta=\epsilon^{-1}\xi^{-1}\rho>1$. Since $\mathcal{G}_u$ is acyclic, it has a topological ordering, say $A_{p_1}, A_{p_2},\dots, A_{p_{\left|\mathcal{P}\right|}}$. Let the ordering be such that $A_{p_1},\dots,A_{p_k}$ are unstable subsystems and have no incoming edges as vertices in the graph $\mathcal{G}_u$, $A_{p_{k+1}},\dots,A_{p_\ell}$ are unstable subsystems, and $A_{p_{\ell+1}},\ldots,A_{p_{\left|\mathcal{P}\right|}}$ are stable subsystems and have no outgoing edges as vertices in the graph $\mathcal{G}_u$. Let $P_{p_j}=\theta^{j-k}Q_{p_j}$, for $j={k+1},\ldots,\ell$, $P_{p_j}=\theta^{j-\ell}Q_{p_j}$, for $j={\ell+1},\ldots,{\left|\mathcal{P}\right|}$, and $P_p=Q_p$ for all other $p\in\mathcal{P}$. Note that for $p\in\mathfrak{u}$, $c_p\|P_q^{-1}MP_p\|\le\theta^{-1}c_p\|Q_q^{-1}MQ_p\|\le\theta^{-1}\rho=\epsilon\xi<\epsilon$, and the result follows.
			\end{proof}
			
			In view of the preceding result, we assume the following hypothesis which will be crucial in obtaining flee time ensuring stability of the switched system:\begin{hypothesis}\label{hyp}
			The subgraph $\mathcal{G}_u$ of $\mathcal{G}$, consisting of edges originating from vertices corresponding to unstable susbsystems, is acyclic. 
			\end{hypothesis}

			Henceforth, throughout this section, we will assume that the reset (impulsive) switched system has switching governed by a graph $\mathcal{G}$ satisfying \Cref{hyp}. Recall the collections $U_{p,q}^\eta$ and $S_{p,q}^\tau$ of matrices as defined in~\eqref{eq:US}. Note that both of these collections are compact subsets of the set $M_n(\mathbb{R})$ of all $n\times n$ real matrices.
			
			The next two results address an extended version of issues~\ref{issue1} and~\ref{issue2} where the choice of resets (impulses) is only allowed to be made from a given compact collection of matrices. Since there is a restriction on available resets (impulses), it is natural to expect certain restrictions on the time constraints so that we are able to stabilize the switched system using resets (impulses) from the given set of compact matrices. We prove that specific time constraints on the signal allow us to choose the set of resets (the set of impulses) from the given collection of matrices, so that the reset (impulsive) switched system is stable. 
			
			\begin{theorem}\label{thm:ConstraintsGivenResets} 
				Consider a switched system~\eqref{eq:system} governed by the graph $\mathcal{G}$ satisfying \Cref{hyp}, and a compact collection $\mathbb{M}$ of real matrices of size $n$ not containing the zero matrix $\mathbf{O}$. Then there exist $\tau^*\ge 0$, $\eta^*> 0$, and a collection $\mathcal{R}\subseteq \mathbb{M}$ of nonzero resets such that the reset switched system~\eqref{eq:RSS} is stable for all switching signals $\sigma\in\mathcal{S}^{\mathcal{G}}[\tau^*,\eta^*]$.
			\end{theorem}
			
			\begin{proof}
				First of all, for each $\eta>0$, we prove that, for each $p\in\mathcal{P}$, there exist Jordan basis matrix $P_p$ (depending on $\eta$) of $A_p$, such that $\mathbb{M}\subseteq \cap_{(p,q)\in\mathcal{E}(\mathcal{G}_u)}U_{p,q}^{\eta}$. This is where the \Cref{hyp} will be used. Fix $\eta>0$ and let $\epsilon=\min_{p\in\mathfrak{u}}\{e^{-\mu_p^*\eta}\}\in (0,1)$. By \Cref{lemma:acyclic_cond}, for each $p\in\mathcal{P}$, there is a Jordan basis matrix $P_p$ (depending on $\epsilon$, and hence $\eta$) of $A_p$ so that $\max\left\{c_p\left\|P_q^{-1}MP_p\right\|\colon (p,q)\in \mathcal{E}(\mathcal{G}_u), M\in\mathbb{M}\right\}<\epsilon$. Hence for each $(p,q)\in \mathcal{E}(\mathcal{G}_u)$ and for each $M\in\mathbb{M}$, we have $\left\|P_q^{-1}MP_p\right\|<c_p^{-1}e^{-\mu_p^*\eta}$. Thus $\mathbb{M}\subseteq \cap_{(p,q)\in  \mathcal{E}(\mathcal{G}_u)}U_{p,q}^{\eta}$. 
				
				In particular, fix a number $\eta_{\min}>0$ and choose Jordan basis matrix $P_p$ (depending on $\eta_{\min}$) of $A_p$, $p\in\mathcal{P}$, such that $\mathbb{M}\subseteq \cap_{(p,q)\in  \mathcal{E}(\mathcal{G}_u)}U_{p,q}^{\eta_{\min}}$. Then define
				\begin{eqnarray*}
					\tau^*&=&\min\{\alpha\ge 0\ \colon\ S_{p,q}^{\alpha}\cap\mathbb{M}\ne\emptyset, \text{ for all } (p,q)\in \mathcal{E}(\mathcal{G}_s)\},\\
					\eta^*&=&\max\{\beta\ge 0\ \colon\ U_{p,q}^{\beta}\cap\mathbb{M}\ne\emptyset, \text{ for all } (p,q)\in \mathcal{E}(\mathcal{G}_u)\}.
				\end{eqnarray*}
				
				Clearly $\eta^*\ge \eta_{\min}>0$. Also, as $\beta$ increases, the collection $U_{p,q}^\beta$ shrinks for each $(p,q)\in \mathcal{E}(\mathcal{G}_u)$, and also $\cap_{k\in\mathbb{N}}U_{p,q}^{k}=\emptyset$. We prove that there exists $\beta\ge 0$ such that $\left(\cap_{(p,q)\in  \mathcal{E}(\mathcal{G}_u)}U_{p,q}^{\beta}\right)\cap\mathbb{M}=\emptyset$. Suppose it is not true and $\left(\cap_{(p,q)\in  \mathcal{E}(\mathcal{G}_u)}U_{p,q}^{\beta}\right)\cap\mathbb{M}\ne\emptyset$ for all $\beta\ge 0$. Then $\{\mathbb{M}\cap \left(\cap_{(p,q)\in \mathcal{E}(\mathcal{G}_s)} U_{p,q}^{k}\cup\{\bf{O}\}\right)\}_{k\in\mathbb{N}}$ is a shrinking sequence of non-empty compact sets, hence their intersection is non-empty by the Cantor's intersection theorem. This leads to a contradiction since $\bf{O}\notin \mathbb{M}$ and $\cap_{k\in\mathbb{N}}U_{p,q}^{k}=\emptyset$. Hence $\eta^*$ is finite.
				
				Furthermore as $\alpha$ increases, the collection $S_{p,q}^\alpha$ expands for each $(p,q)\in \mathcal{E}(\mathcal{G}_s)$, and also $\cup_{\alpha\ge 0}S_{p,q}^{\alpha}=M_n(\mathbb{R})\setminus \{\bf{O}\}$. Since $\mathbb{M}$ is compact and does not contain $\bf{O}$, there exists $\tau_{\max}\ge 0$ such that $\mathbb{M}\subseteq \cap_{(p,q)\in \mathcal{E}(\mathcal{G}_s)}S_{p,q}^{\tau_{\max}}$. Clearly $0\le \tau^*\le \tau_{\max}$. 
				
				By the arguments given thus far, the reset switched system~\eqref{eq:RSS} with signal $\sigma\in\mathcal{S}^{\mathcal{G}}[\tau^*,\eta^*]$ is stable if $R_{(p,q)}\in S_{p,q}^{\tau^*}\cap\mathbb{M}$ when $(p,q)\in \mathcal{E}(\mathcal{G}_s)$, and $R_{(p,q)}\in U_{p,q}^{\eta^*}\cap\mathbb{M}$ when $(p,q)\in \mathcal{E}(\mathcal{G}_u)$.

				The following statements are equivalent leading to the expressions which can be used to compute $\tau^*$. 
				\begin{itemize}
					\item $S_{p,q}^{\alpha}\cap\mathbb{M}\ne\emptyset$ for each $(p,q)\in \mathcal{E}(\mathcal{G}_s)$.
					\item There exists $M_{(p,q)}\in\mathbb{M}$ such that $M_{(p,q)}\in S_{p,q}^{\alpha}$ for each $(p,q)\in \mathcal{E}(\mathcal{G}_s)$ which is equivalent to $\alpha\ge \frac{\ln \left(c_p\, \|P_q^{-1}M_{(p,q)}P_p\|\right)}{\lambda_p^*}$, for all $(p,q)\in \mathcal{E}(\mathcal{G}_s)$.
					\item $\alpha\ge \max_{(p,q)\in \mathcal{E}(\mathcal{G}_s)}\min_{M\in\mathbb{M}}\frac{\ln \left(c_p\, \|P_q^{-1}MP_p\|\right)}{\lambda_p^*}$.
				\end{itemize} 
				Note that the minimum over $\mathbb{M}$ (not containing $\bf{O}$) is attained since the objective function is continuous and $\mathbb{M}$ is compact. Hence we obtain
				\begin{eqnarray*}
					\tau^*=\max_{(p,q)\in\mathcal{E}(\mathcal{G}_s)}\,\min_{M\in\mathbb{M}}\,\frac{\ln\left( c_p\|P_q^{-1}MP_p\|\right)}{\lambda_p^*}.
				\end{eqnarray*}
				Using similar arguments, we obtain
				\begin{eqnarray*}
					\eta^*=-\max_{(p,q)\in\mathcal{E}(\mathcal{G}_u)}\,\min_{M\in\mathbb{M}}\,\frac{\ln \left(c_p\|P_q^{-1}MP_p\|\right)}{\mu_p^*}.
				\end{eqnarray*}
			\end{proof}
			
			Consequently, if we have a compact set of possible resets $\mathbb{M}$ not containing $\bf{O}$ and any signal $\sigma\in\mathcal{S}^{\mathcal{G}}[\tau,\eta]$ with $\tau\ge \tau^*$ and $\eta\le\eta^*$, there exists a collection of nonzero reset maps $\mathcal{R}\subseteq \mathbb{M}$ so that the reset switched system $\Sigma_{\sigma}^{\mathcal{R}}$ is stable. Also note that $\eta^*$ is dependent on the initial choice of $\epsilon$, hence can be made arbitrary large or small, and yields a corresponding dwell time $\tau^*$. So there is a relation (through the choice of Jordan basis matrices $P_p$) between $\tau^*$ and $\eta^*$. Such relations exist in the literature in the context of stability of switched systems with both stable and unstable subsystems, and are referred to as \textit{dwell-flee relations}~\cite{tripathi2022unistable}. The preceding theorem, thus, states that switched systems satisfying the above dwell-flee relations (dwell time constraints, when there are no unstable subsystems) allow us to choose resets from the given compact set of matrices so that the reset switched system is stable. We now address the extended version of issue~\ref{issue2}.

			\begin{theorem}
				Consider a switched system~\eqref{eq:system} governed by the graph $\mathcal{G}$, and a compact set $\mathbb{M}$ of real matrices of size $n$ not containing the zero matrix $\bf{O}$. Then there exist $\tau^*\ge 0$, $\eta^*> 0$, and a collection $\mathcal{I}\subseteq \mathbb{M}$ of nonzero impulses such that the impulsive switched system~\eqref{eq:ISS} is stable under arbitrary impulses for all switching signals $\sigma\in\mathcal{S}^{\mathcal{G}}[\tau^*,\eta^*]$.\end{theorem} 
			\begin{proof}
				Referring to the proof of \Cref{thm:ConstraintsGivenResets}, we have $\eta_{\min}\ge 0$, $\tau_{\max}\ge 0$ and Jordan basis matrix $P_p$ of $A_p$ for each $p\in\mathcal{P}$ such that $\mathbb{M}\subseteq \cap_{(p,q)\in \mathcal{E}(\mathcal{G}_s)}S_{p,q}^{\tau_{\max}}$ and $\mathbb{M}\subseteq \cap_{(p,q)\in \mathcal{E}(\mathcal{G}_u)}U_{p,q}^{\eta_{\min}}$. This shows existence of $\tau^*$ and $\eta^*$. The set of impulses can be taken to be $\mathcal{I}=\mathbb{M}$ if $\tau^*=\tau_{\max}$ and $\eta^*=\eta_{\min}$. But since $ \cap_{(p,q)\in\mathcal{E}(\mathcal{G}_u)}U_{p,q}^{\beta}$ shrinks as $\beta$ increases and $ \cap_{(p,q)\in\mathcal{E}(\mathcal{G}_s)}S_{p,q}^{\alpha}$ expands as $\alpha$ increases, it is possible to choose $\tau^*<\tau_{\max}$ and $\eta^*>\eta_{\min}$ such that $\mathbb{M}\bigcap \left(\cap_{(p,q)\in\mathcal{E}(\mathcal{G}_u)}U_{p,q}^{\eta^*}\right)\bigcap \left(\cap_{(p,q)\in\mathcal{E}(\mathcal{G}_s)}S_{p,q}^{\tau^*}\right)\ne\emptyset$, and then we choose $\mathcal{I}=\mathbb{M}\bigcap \left(\cap_{(p,q)\in\mathcal{E}(\mathcal{G}_u)}U_{p,q}^{\eta^*}\right)\bigcap \left(\cap_{(p,q)\in\mathcal{E}(\mathcal{G}_s)}S_{p,q}^{\tau^*}\right)$.    
			\end{proof}

				The next two results address issues~\ref{issue3} and~\ref{issue4}, where for a given reset/impulsive switched system, we obtain time constraints on signals so that the reset/impulsive switched system is stable.
				
				\begin{theorem}\label{thm:DTforRSS_flow}
					The reset switched system~\eqref{eq:RSS} governed by the graph $\mathcal{G}$ is stable for all switching signals $\sigma\in\mathcal{S}^{\mathcal{G}}[\tau_R,\eta_R]$, where \[
					\tau_R=\max_{(p,q)\in\mathcal{E}(\mathcal{G}_s)}\frac{\ln \left(c_p\, \|P_q^{-1}R_{(p,q)}P_p\|\right)}{\lambda_p^*},\ \ \ \eta_R=-\max_{(p,q)\in\mathcal{E}(\mathcal{G}_u)}\frac{\ln \left(c_p\, \|P_q^{-1}R_{(p,q)}P_p\|\right)}{\mu_p^*}.
					\]
				\end{theorem}
				
				\begin{proof}
					The proof follows along the lines of the proof of \Cref{thm:ResetsGivenConstraints}. 
				\end{proof}
				
				\begin{theorem}\label{thm:ISS:flow}
					The impulsive switched system~\eqref{eq:ISS} governed by the graph $\mathcal{G}$ and having a compact set of impulses $\mathcal{I}$ is stable under arbitrary impulses for all switching signals $\sigma\in\mathcal{S}^{\mathcal{G}}[\tau_I,\eta_I]$, where \[
					\tau_I=\max_{(p,q)\in\mathcal{E}(\mathcal{G}_s)}\max_{M\in\mathcal{I}}\frac{\ln \left(c_p\, \|P_q^{-1}M P_p\|\right)}{\lambda_p^*},\ \ \ \eta_I=-\max_{(p,q)\in\mathcal{E}(\mathcal{G}_u)}\max_{M\in\mathcal{I}}\frac{\ln \left(c_p\, \|P_q^{-1}M P_p\|\right)}{\mu_p^*}.
					\]
				\end{theorem}
				
				\begin{proof}
					Let $\sigma\in\mathcal{S}[\tau_I,\eta_I]$. If $p\in\mathfrak{u}$, the time spent in the subsystem $A_p$, say $\eta$, satisfies $\eta\le\eta_I$. Thus, we have $\|P_q^{-1}MP_p\|\le c_p^{-1}{\rm{e}}^{-\mu_p^*\eta}$ for all $M\in\mathcal{I}$. Similarly, if $p\in\mathfrak{s}$, the time spent in the subsystem $A_p$, say $\tau$, satisfies $\tau\ge\tau_I$. Then $\|P_q^{-1}MP_p\|\le c_p^{-1}{\rm{e}}^{\lambda_p^*\tau}$ for all $M\in\mathcal{I}
					$. Hence the result follows. 		
					Note that for each $p,q\in\mathcal{P}$, the functions $f_{p,q}(M)= \ln \left(c_p\, \|P_q^{-1}M P_p\|\right)/\lambda_p^*$ and $g_{p,q}(M)= -\ln \left(c_p\, \|P_q^{-1}M P_p\|\right)/\mu_p^*$ are continuous, hence attain maximum and minimum, respectively, in compact domains. Thus the definitions of $\tau_I$ and $\eta_I$ make sense. 
				\end{proof}


					%

					\section{Using Multiple Lyapunov Functions}\label{sec:MLF}
					
					A standard way of proving stability of an $n$-dimensional system $\dot{x}=f(x)$, where $f:\mathbb{R}^n\rightarrow \mathbb{R}^n$, is to find a positive definite $\mathcal{C}^1$-function $V\colon\mathbb{R}^n\to\mathbb{R}$, known as a \textit{Lyapunov function}, which decreases along any solution trajectory of the system, that is, $\dot{V}(x(t))<0$, for each initial condition $x(0)\in\mathbb{R}^n$. The linear system $\dot{x}=Ax$ is stable if and only if $A$ is a Hurwitz matrix. For a symmetric positive definite matrix $Q$, the function $V(x)=x^\top Q x$ is a Lyapunov function for the linear system $\dot{x}=A x$ if and only if $QA+A^\top Q$ is negative definite. A known extension of this method to address the stability of switched systems is known as \textit{Multiple Lyapunov functions} (MLFs) approach, which will be used in this section. 
					
					\subsection{When all subsystems are stable}\label{sec:MLF:allstable}
					
					For a given switched system, if all the subsystems share a common Lyapunov function, then it serves as a Lyapunov function for the switched system as well, and stability of the switched system under arbitrary switching is guaranteed. Existence of a common Lyapunov function is not always possible. The MLF approach for stability of the switched system~\eqref{eq:system} is a two-step process: \begin{enumerate}
						\item For each $p\in\mathcal{P}$, obtain a quadratic Lyapunov function $V_p(x)=x^\top Q_p x$ for each subsystem $\dot{x}=A_p x$. For a given signal $\sigma$, construct the piecewise Lyapunov function along the trajectory as $V(x(t)):=V_{\sigma(t)}(x(t))$, for $t\ge 0$, which switches between to the Lyapunov function $V_p$ in accordance with the active subsystem $\dot{x}=A_px$. By construction, $V$ decreases along any solution trajectory in between any two consecutive switching instances, that is, for all $k\ge 1$ and for any $x(0)\in\mathbb{R}^n$, $\dot{V}(x(t))<0$ when $t\in (t_{k-1},\,t_{k})$. However, the function $V$ is, in general, discontinuous, see~\cite{liberzon2003switching}.
						\item Relax the condition of $V$ decreasing along any solution trajectory by $V(x(t))$ decreasing along the sequence of switching instances $\{t_k\}_{k\ge 0}$ for any initial condition $x(0)\in\mathbb{R}^n$.	
					\end{enumerate}
					
					Thus the stability results for linear systems are obtained in terms of certain Linear Matrix Inequalities (LMIs). For instance, the following result by Hespanha and Morse in~\cite{hespanha2002switching} provides a sufficient condition for stability of the reset switched system~\eqref{eq:RSS} under arbitrary switching.
					
					\begin{theorem}[Hespanha and Morse~\cite{hespanha2002switching}]
						Suppose there exists a collection of symmetric, positive definite matrices $\{Q_p\in\,M_n(\mathbb{R})\colon\ p\in\mathcal{P}\}$ such that \begin{eqnarray*}
							Q_pA_p+A_p^\top Q_p<0, \text{ for all }\, p\in\mathcal{P}, \text{ and }
							R_{(q,p)}^\top Q_p R_{(q,p)}\le  Q_q, \text{ for all }\,p,q\in\mathcal{P},
						\end{eqnarray*}
						then the reset switched system~\eqref{eq:RSS} is stable under arbitrary switching.
					\end{theorem}
					
					Another result of our interest is the following result by Geromel and Colaneri in~\cite{geromel2006stability} which gives a sufficient condition for stability of switched system~\eqref{eq:system} under a dwell time switching strategy.
					
					\begin{theorem}[Geromel and Colaneri~\cite{geromel2006stability}]
						Assume that for some $\tau>0$, there exists a collection of symmetric, positive definite matrices $\{Q_p\in\,M_n(\mathbb{R})\colon\ p\in\mathcal{P}\}$ such that \begin{eqnarray*}
							Q_pA_p+A_p^\top Q_p< 0, \text{ for all }\, p\in\mathcal{P}, \text{ and } {\rm{e}}^{A_q^\top\tau} Q_p {\rm{e}}^{A_q\tau}\le  Q_q, \text{ for all }\,p,q\in\mathcal{P},
						\end{eqnarray*}
						then the switched system~\eqref{eq:system} is stable for all signals $\sigma\in\mathcal{S}_s[\tau]$.
					\end{theorem}
					
					The above two results motivate us to obtain a sufficient condition for stability of the reset switched system~\eqref{eq:RSS} under dwell time switching strategy, thereby addressing issue~\ref{issue3}. This condition is given in the following result.
					
					\begin{theorem}\label{thm:dtRSS}
						Given the reset switched system~\eqref{eq:RSS}, assume that for some $\tau>0$, there exists a collection of symmetric, positive definite matrices $\{Q_p\in\,M_n(\mathbb{R})\colon\ p\in\mathcal{P}\}$ such that \begin{eqnarray}
							Q_pA_p+A_p^\top Q_p&<&0\,\,\,\,\,\text{ for all }\, p\in\mathcal{P},\label{eq:mlf1rss}\\
							{\rm{e}}^{A_q^\top\tau}R_{(q,p)}^\top Q_p R_{(q,p)}{\rm{e}}^{A_q\tau}&\le& Q_q\,\text{ for all }\,p,q\in\mathcal{P},\label{eq:mlf2rss}
						\end{eqnarray}
						then the reset switched system~\eqref{eq:RSS} is stable for all signals $\sigma\in\mathcal{S}_s[\tau]$.
					\end{theorem}
					
					\begin{proof}
						Due to~\eqref{eq:mlf1rss}, there exists a small enough $\lambda>0$ such that for all $p\in\mathcal{P}$, we have $Q_pA_p+A_p^\top Q_p<-2\lambda Q_p$. Consider a signal $\sigma\in\mathcal{S}_s[\tau]$. Let $V_p(x)=x^\top Q_px$, and define $V(x)=V_{\sigma}(x)$. Then using~\eqref{eq:mlf1rss}, we have $\dot{V}(x(t))=\dot{V_q}(x(t))=x(t)^\top \left(A_q^\top Q_q+Q_q A_q\right)x(t)<-2\lambda V_q(x(t))$, when $\sigma(t)=q$. If $\sigma$ switches from $p=\sigma(t_{k+1}^-)$ to $q=\sigma(t_{k+1})$, then \begin{eqnarray*}
							V_q(x(t_{k+1}))&=&x(t_{k+1})^\top\, Q_q\,x(t_{k+1})=x(t_{k+1}^-)^\top R_{(p,q)}^\top\, Q_q\,R_{(p,q)}x(t_{k+1}^-)\\
							&=& x(t_{k})^\top {\rm{e}}^{A_p^\top(t_{k+1}-t_k)}R_{(p,q)}^\top\, Q_q\,R_{(p,q)}{\rm{e}}^{A_p(t_{k+1}-t_k)}x(t_{k})\le  x(t_{k})^\top Q_px(t_{k})\\
							&=&V_p(x(t_k)),
						\end{eqnarray*} 
						where the inequality follows due to~\eqref{eq:mlf2rss} under the assumption that $t_{k+1}-t_k\ge \tau$. It follows that $V_{\sigma(t_{j})}(x(t_{j}))\le V_{\sigma(t_{i})}(x(t_{i}))$ for all $j\ge i$. Even though $V$ may be discontinuous at the discontinuities of $\sigma$, the following inequality is satisfied for all $t\in[t_i,t_{i+1})$: $V_{\sigma(t)}(x(t))\le {\rm{e}}^{-2\lambda(t-t_i)}V_{\sigma(t_i)}(x(t_i))\le {\rm{e}}^{-2\lambda(t-t_i)}V_{\sigma(t_0)}(x(t_0))\le V_{\sigma(t_0)}(x(t_0))$. Since $\lambda_{\min}(Q_p) \|x\|^2\le x^\top Q_p x\le \lambda_{\max}(Q_p) \|x\|^2$ for each $p\in\mathcal{P}$, we have $\|x(t)\|\le\sqrt{\frac{\max_p \lambda_{\max}(Q_p)}{\min_p \lambda_{\min}(Q_p)}}\,\|x(0)\|$, due to the equivalence of norms on $\mathbb{R}^n$ and the result follows.
					\end{proof}
					
					The above proof can be adapted to obtain the following result addressing issue~\ref{issue4}.
					
					\begin{theorem}\label{thm:MLF:issue5}
						Given the impulsive switched system~\eqref{eq:ISS} with a compact set of impulses $\mathcal{I}$, assume that for some $\tau>0$, there exists a collection of symmetric, positive definite matrices $\{Q_p\in\,M_n(\mathbb{R})\colon\ p\in\mathcal{P}\}$  such that \begin{eqnarray}
							Q_pA_p+A_p^\top Q_p&<&0,\,\,\,\,\,\text{ for all }\, p\in\mathcal{P},\label{eq:mlf1cpt}\\
							{\rm{e}}^{A_q^\top\tau}M^\top Q_p M{\rm{e}}^{A_q\tau}&\le& Q_q,\,\text{ for all }\,p,q\in\mathcal{P}\text{ and }M\in\mathcal{I},\label{eq:mlf2cpt}
						\end{eqnarray}
						are satisfied. Then the impulsive switched system~\eqref{eq:ISS} is stable under arbitrary impulses for all signals $\sigma\in\mathcal{S}_s[\tau]$.
					\end{theorem}
					
					\begin{proof}
						The proof follows from that of \Cref{thm:dtRSS} by using conditions~\eqref{eq:mlf1cpt} and~\eqref{eq:mlf2cpt} instead.
					\end{proof}
					%
					%
					%
					%
					
					\subsection{When some subsystems are unstable}
					The methods used in \Cref{sec:MLF:allstable} can not be extended to switched systems containing unstable subsystems. For such systems, we use Lyapunov-like functions for the unstable subsystems. These functions are allowed to increase along the solution trajectories of the unstable subsystems but with a bounded increase rate. The increment caused by the unstable subsystems is compensated by the decrement produced by stable subsystems. 
						As earlier, let $\{t_k\}_{k\ge 1}$ denote the increasing sequence of discontinuities of $\sigma$. Define \begin{eqnarray*}
							N_s(0,t)=\#\{i\colon\, 0\le t_i< t,\,\sigma(t_i)\in\mathfrak{s}\} &,& N_u(0,t)=\#\{i\colon\, 0\le t_i< t,\,\sigma(t_i)\in\mathfrak{u}\},\\
							N(0,t)&=&N_s(0,t)+N_u(0,t),
						\end{eqnarray*}
						which is the number switches of $\sigma$ to stable subsystems, number of switches of $\sigma$ to unstable subsystems and total number of switches of $\sigma$ in the time interval $(0,t)$, respectively. With these notations, we address the issue~\ref{issue3}.

						\begin{theorem}\label{thm:MLF:mixed}
							Given the reset switched system~\eqref{eq:RSS}, suppose there exist a collection of symmetric, positive definite matrices $\{Q_p\in\,M_n(\mathbb{R})\colon\ p\in\mathcal{P}\}$, positive constants $\lambda_p$ for $p\in \mathfrak{s}$, $\mu_p$ for $p\in\mathfrak{u}$, and $\gamma_{(p,q)}$, for $p,q\in\mathcal{P}$ satisfying \begin{eqnarray*}
								Q_pA_p+A_p^\top Q_p\le -\lambda_p Q_p, \text{ for all }\,p\in\mathfrak{s}&,& 
								Q_pA_p+A_p^\top Q_p\le \mu_p Q_p, \text{ for all }\,p\in\mathfrak{u},\\
								R_{(p,q)}^\top Q_qR_{(p,q)}&\le& \gamma_{(p,q)}Q_p,\hspace{17pt}\text{ for all }\,p,q\in\mathcal{P}.
							\end{eqnarray*}
							Let $\lambda=\min_{p\in\mathfrak{s}}\,\{\lambda_p\}$, $\mu=\max_{p\in\mathfrak{u}}\,\{\mu_p\}$ and $\gamma=\max_{p,q\in\mathcal{P}}\,\{\gamma_{(p,q)}\}$. Then for any signal $\sigma\in\mathcal{S}[\tau,\eta]$ satisfying 
							\begin{eqnarray}\label{eq:limsup}
								\lim\sup_{t\to\infty}\left(-\lambda\, N_s(0,t)\,\tau+\mu\, N_u(0,t) \,\eta+N(0,t)\,\ln\gamma\right)<0, 
							\end{eqnarray}
							the reset switched system~\eqref{eq:RSS} is stable. 
						\end{theorem}
						
						\begin{remark}
							Observe that if all the subsystems are unstable, then $\gamma<1$ is a necessary condition for~\eqref{eq:limsup} to be satisfied. The condition reduces to $\gamma<e^{-\mu\eta}$. 
						\end{remark}
						
						\begin{proof}
							Consider a signal $\sigma\in\mathcal{S}[\tau,\eta]$. Define a function $V(x)=V_{\sigma}(x(t))$, where $V_p(x)=x^\top Q_p x$. Assume $p\in\mathfrak{s}$ (similar arguments for the case $p\in\mathfrak{u}$) is the subsystem active in $[t_0,t_1)$, we have $V_p(x(t))\le{\rm{e}}^{-\lambda_p(t-t_0)}V_p(x(t_0))$, for all $t\in[t_0,t_1)$. This implies that $V_p(x(t_1^-))\le {\rm{e}}^{-\lambda_p(t_1-t_0)}V_p(x(t_0))$. Suppose that $\sigma(t_1)=q$, then we have \begin{eqnarray*}
								V_q(x(t_1))=V_q(R_{(p,q)}x(t_1^-))=x^\top(t_1^-)R_{(p,q)}^\top Q_q R_{(p,q)}x(t_1^-)&\le& \gamma_{(p,q)}x^\top(t_1^-)Q_px(t_1^-)\\
								&=&\gamma_{(p,q)}V_p(x(t_1^-)).
							\end{eqnarray*}
							Thus, we have $V_q(x(t_1))\le\gamma_{(p,q)}{\rm{e}}^{-\lambda_p(t_1-t_0)}V_p(x(t_0))$. Similarly if $\sigma(t_2)=r$, we have \begin{eqnarray*}
								V_r(x(t_2))\le\left\{\begin{aligned}
									&\gamma_{(p,q)}\gamma_{(q,r)}{\rm{e}}^{-\lambda_p(t_1-t_0)}{\rm{e}}^{\mu_q(t_2-t_1)}\,V_p(x(t_0)), &\text{if }q\in\mathfrak{u}\\
									&\gamma_{(p,q)}\gamma_{(q,r)}{\rm{e}}^{-\lambda_p(t_1-t_0)}{\rm{e}}^{-\lambda_q(t_2-t_1)}\,V_p(x(t_0)), &\text{ if }q\in\mathfrak{s}.
								\end{aligned}\right.
							\end{eqnarray*}
							Continuing in a similar manner, we have \begin{equation}\label{eq:MLF:BoundLyap}
								V_{\sigma(t_k)}(x(t_k))\le \gamma^{N(0,t_k)}{\rm{e}}^{-\lambda\, N_s(0,t_k)\, \tau+\mu\, N_u(0,t_k)\,\eta}\, V_{\sigma(t_0)}(x(t_0)).
							\end{equation}
							
							Now if $\lim\sup_{k\to\infty}\left(-\lambda\, N_s(0,t_k)\,\tau+\mu\, N_u(0,t_k) \,\eta+N(0,t_k)\,\ln\gamma\right)<0$, stability of the reset switched system~\eqref{eq:RSS} is guaranteed.
						\end{proof}
						
						%
						
						We now present the final result of this section addressing the issue~\ref{issue4}.
						
						\begin{theorem}\label{thm:MLF:mixed:arbitrary}
							Given an impulsive switched system~\eqref{eq:ISS} with a compact set of impulses $\mathcal{I}$. Let there exist a collection of symmetric, positive definite matrices $\{Q_p\in\,M_n(\mathbb{R})\colon\ p\in\mathcal{P}\}$, constants $\lambda_p>0$ for $p\in \mathfrak{s}$, $\mu_p>0$ for $p\in\mathfrak{u}$, and $\gamma_{(p,q)}\ge 1$ for $p,q\in\mathcal{P}$ satisfying \begin{eqnarray*}
								Q_pA_p+A_p^\top Q_p&\le&-\lambda_p Q_p,\hspace{20pt}\text{for all }\,p\in\mathfrak{s},\\
								Q_pA_p+A_p^\top Q_p&\le&\mu_p Q_p,\hspace{28pt}\text{for all }\,p\in\mathfrak{u},\\
								M^\top Q_pM&\le& \gamma_{(q,p)}Q_q,\hspace{17pt}\text{for all }\,p,q\in\mathcal{P}\text{ and }M\in\mathcal{I}.
							\end{eqnarray*}
							Let $\lambda=\min_{p\in\mathfrak{s}}\,\{\lambda_p\}$, $\mu=\max_{p\in\mathfrak{u}}\,\{\mu_p\}$, and $\gamma=\max_{p,q\in\mathcal{P}}\,\{\gamma_{(p,q)}\}$. Then for any signal $\sigma\in\mathcal{S}[\tau,\eta]$ satisfying $\lim\sup_{t\to\infty}\left(-\lambda\, N_s(0,t)\,\tau+\mu\, N_u(0,t) \,\eta+N(0,t)\,\ln\gamma\right)<0$, the impulsive switched system~\eqref{eq:ISS} is stable under arbitrary impulses. 
						\end{theorem}
						
						\begin{proof}
							The given conditions imply that~\eqref{eq:MLF:BoundLyap} is satisfied for all $\omega$.
						\end{proof}
						
						\section{Convex hull of resets}\label{sec:convexhull}
						
						The dwell time constraints and dwell-flee relations for stability under arbitrary impulses obtained thus far have been in terms of infinitely many linear matrix inequalities when the collection of impulses is infinite (but compact), and standard LMI tools can not be applied here. However, these infinitely many inequalities reduce to only finitely many when the collection of impulses is a convex hull of finitely many matrices. Suppose we have a switched system~\eqref{eq:system} and the collection of impulses is given by a convex hull of $k$ matrices $M_1,\ldots,M_k\in M_n(\mathbb{R})$, that is, $\mathcal{I}=\text{conv}\{M_1,\ldots,M_k\}=\left\{\sum_{i=1}^k\gamma_iM_i\colon\ \gamma_i\ge 0\,\text{ for all }i,\text{ and }\sum_{i=1}^k\gamma_i=1\right\}$, Since the obtained LMIs are convex with respect to the impulse matrix variable, the LMIs can be reduced to finitely many by using the fact that the maximum value of a convex function over this convex hull is obtained at one of the vertices $M_i$,~\cite{jeter1986mathematical}. Specifically,
						\begin{enumerate}
							\item In \Cref{thm:MLF:issue5}, condition~\eqref{eq:mlf2cpt} consists of infinitely many inequalities if $\mathcal{I}$ is an infinite set. Suppose $\mathcal{I}=\text{conv}\{M_1,\ldots,M_k\}$. Denoting $\|x\|_{Q_p}=\sqrt{x^\top Q_p x}$, the inequality~\eqref{eq:mlf2cpt} can be rewritten as \[
							\|M{\rm{e}}^{A_q\tau}\,x\|_{Q_p}\le \|x\|_{Q_q}\hspace{10pt}\forall\,x\in\mathbb{R}^n,\,\forall p,q\in\mathcal{P},\,\forall M\in\,\mathcal{I}.
							\]
							For each $x\in\mathbb{R}^n$, the function $f_q(M)=\|M{\rm{e}}^{A_q\tau}\,x\|_{Q_p}$ is convex and hence attains its maximum at one of the vertices of the convex hull $\mathcal{I}$. Thus the above set of inequalities are satisfied for all $M\in\mathcal{I}$ if and only if these are satisfied for $M=M_1,\ldots,M_k$ only. Thus, the inequalities~\eqref{eq:mlf2cpt} reduce to $k\lvert \mathcal{P}\rvert\,\left(\lvert\mathcal{P}\rvert-1\right)$ many \[
							\|M_i{\rm{e}}^{A_q\tau}\,x\|_{Q_p}\le \|x\|_{Q_q}\hspace{10pt}\forall\,x\in\mathbb{R}^n,\,\forall p,q\in\mathcal{P},\,\forall i=1,\ldots,k.
							\]
							Thus in \Cref{thm:MLF:issue5}, condition~\eqref{eq:mlf2cpt} can be replaced by 
							\[
							{\rm{e}}^{A_q^\top\tau}M_i^\top Q_p M_i{\rm{e}}^{A_q\tau}\le Q_q\,\text{ for all }p,q\in\mathcal{P}\text{ and }i=1,\ldots,k.
							\]
							\item Using similar arguments, if $\mathcal{I}=\text{conv}\{M_1,\ldots,M_k\}$, in \Cref{thm:MLF:mixed:arbitrary}, the third set of inequalities $M^\top Q_pM\le \gamma_{(q,p)}Q_q$, $\text{for all }\,p,q\in\mathcal{P}\text{ and }M\in\mathcal{I}$, can be replaced by $k\lvert \mathcal{P}\rvert\,\left(\lvert\mathcal{P}\rvert-1\right)$ many inequalities given below: 
							\[
							M_i^\top Q_pM_i\le \gamma_{(q,p)}Q_q,\hspace{17pt}\text{for all }\,p,q\in\mathcal{P}\text{ and }i=1,\ldots,k.
							\]
							\item If $P_p,P_q\in M_n(\mathbb{R})$, the function $f_{p,q}(M)=\|P_q^{-1}MP_p\|$ is convex. Thus, the maximum value of the function over the convex hull $\mathcal{I}=\text{conv}\{M_1,\ldots,M_k\}$ is attained at one of the vertices $M_i$. This can be used to simplify the constraints for stability under arbitrary impulses obtained in \Cref{thm:ISS:flow}. The expressions for $\tau_I$ and $\eta_I$ can be reduced to 
							\begin{eqnarray*}
								\tau_I&=&\max_{(p,q)\in\mathcal{E}(\mathcal{G}_s)}\max_{i=1,\dots,k}\frac{\ln \left(c_p\, \|P_q^{-1}M_i P_p\|\right)}{\lambda_p^*},\\ \eta_I&=&-\max_{(p,q)\in\mathcal{E}(\mathcal{G}_u)}\max_{i=1,\dots,k}\frac{\ln \left(c_p\, \|P_q^{-1}M_i P_p\|\right)}{\mu_p^*}.
							\end{eqnarray*}
						\end{enumerate}

						\section{Mode-Dependent Constraints}\label{sec:MDDT}
						
						Finding dwell time constraint is one of the most common ways to establish stability of a given switched system when all subsystems are stable. However, the requirement of spending more than a fixed amount of time after every switching instance can be relaxed by allowing different dwell time constraints depending on the active subsystem (or mode). As a result, a larger class of stabilizing signals can be obtained, refer to \Cref{eg:bplssbistab}. This also allows us to select the switching times `on the fly' to ensure stability, if the information of the active mode is known as a feedback at every switching instance. Mode-dependent dwell time constraints have been used in literature to stabilize switched systems both in the continuous-time setting~\cite{briatMDDT}, and discrete-time setting~\cite{DEHGHAN20131804}. The more general \textit{mode-dependent average dwell time} constraints have also been widely used, in the literature, to stabilize switched systems with both stable and unstable subsystems, for instance~\cite{zhai2001stability}.
						
						Consider a switched system~\eqref{eq:system} and suppose the increasing sequence of switching times of the signal $\sigma$ is given by $\{t_k\}_{k\ge 1}$. Then $\sigma$ is said to have \textit{mode-dependent dwell times} $\{\tau_p\}_{p\in\mathfrak{s}}$ if $t_{k+1}-t_k\ge \tau_p$, whenever $\sigma(t_k)=p$. Further $\sigma$ is said to have \textit{mode-dependent flee times} $\{\eta_p\}_{p\in\mathfrak{u}}$ if $t_{k+1}-t_k\le \eta_p$, whenever $\sigma(t_k)=p$. Note that, we have earlier looked at situations when the dwell time and flee time are independent of the mode. Mode-dependent constraints allow for more effective stabilization of switched systems. Depending on the properties of individual subsystems, it is possible to design a slow-fast switching strategy between subsystems. We define the classes of signals having mode-dependent dwell times and mode-dependent flee times as follows:
						\begin{eqnarray*}
							\mathcal{S}_s^{M}\left(\{\tau_p\}_{p\in\mathcal{P}}\right)&=&\left\{\sigma\colon[0,\infty)\to\mathcal{P}\colon\ t_{k+1}-t_k\ge \tau_{\sigma(t_k)}\text{ for all }k\in\mathbb{N}\right\},\\
							\mathcal{S}_u^{M}\left(\{\eta_p\}_{p\in\mathcal{P}}\right)&=&\left\{\sigma\colon[0,\infty)\to\mathcal{P}\colon\ t_{k+1}-t_k\le \eta_{\sigma(t_k)}\text{ for all }k\in\mathbb{N}\right\}\\
							\mathcal{S}^{M}\left(\{\tau_p\}_{p\in\mathfrak{s}},\{\eta_p\}_{p\in\mathfrak{u}}\right)&=&\left\{\sigma\colon[0,\infty)\to\mathcal{P}\colon\ t_{k+1}-t_k\ge \tau_{\sigma(t_k)}\text{ if }\sigma(t_k)\in\mathfrak{s},\right.\\
							& &\hspace{93pt}\left. t_{k+1}-t_k\le \eta_{\sigma(t_k)}\text{ if }\sigma(t_k)\in\mathfrak{u}\right\}.
						\end{eqnarray*}
						These signal classes are generalizations of signal classes defined earlier in this paper. If $\tau_p=\tau$ for all $p\in\mathfrak{s}$ and $\eta_p=\eta$ for all $p\in\mathfrak{u}$, then $\mathcal{S}_s^{M}\left(\{\tau_p\}_{p\in\mathcal{P}}\right)=\mathcal{S}_s[\tau]$, $\mathcal{S}_u^{M}\left(\{\eta_p\}_{p\in\mathcal{P}}\right)=\mathcal{S}_u[\eta]$, and $\mathcal{S}^{M}\left(\{\tau_p\}_{p\in\mathfrak{s}},\{\eta_p\}_{p\in\mathfrak{u}}\right)=\mathcal{S}\left(\tau,\eta\right)$.
						
						The results obtained earlier involving dwell time and flee time can be generalized to involve mode-dependent dwell time and mode-dependent flee time. Here we just state them without proof. The results follow along the lines of proofs of the earlier statements.
						
						\begin{theorem}[Generalization of \Cref{thm:DTforRSS_flow}]\label{thm:gen:DTforRSS}
							The reset switched system~\eqref{eq:RSS} governed by graph $\mathcal{G}$ is stable for all $\sigma\in\mathcal{S}^{\mathcal{G}}\cap \mathcal{S}^{M}\left(\{\tau_p\}_{p\in\mathfrak{s}},\{\eta_p\}_{p\in\mathfrak{u}}\right)$, where \[
							\tau_p=\max_{
								q\colon (p,q)\in\mathcal{E}(\mathcal{G})}\frac{\ln \left(c_p\, \|P_q^{-1}R_{(p,q)}P_p\|\right)}{\lambda_p^*},\ \ \ \eta_p=-\max_{
								q\colon (p,q)\in\mathcal{E}(\mathcal{G})}\frac{\ln \left(c_p\, \|P_q^{-1}R_{(p,q)}P_p\|\right)}{\mu_p^*}.
							\]
						\end{theorem}
						
						\begin{theorem}[Generalization of \Cref{thm:ISS:flow}]
							The impulsive switched system~\eqref{eq:ISS} governed by the graph $\mathcal{G}$ and having a compact set of impulses $\mathcal{I}$ is stable under arbitrary impulses for all $\sigma\in\mathcal{S}^{\mathcal{G}}\cap \mathcal{S}^{M}\left(\{\tau_p\}_{p\in\mathfrak{s}},\{\eta_p\}_{p\in\mathfrak{u}}\right)$, where \[
							\tau_p=\max_{q\colon (p,q)\in\mathcal{E}(\mathcal{G})}\max_{M\in\mathcal{I}}\frac{\ln \left(c_p\, \|P_q^{-1}M P_p\|\right)}{\lambda_p^*},\  \eta_p=-\max_{q\colon (p,q)\in\mathcal{E}(\mathcal{G})}\max_{M\in\mathcal{I}}\frac{\ln \left(c_p\, \|P_q^{-1}M P_p\|\right)}{\mu_p^*}.
							\]
						\end{theorem}
						
						\begin{theorem}[Generalization of \Cref{thm:dtRSS}]
							Given the reset switched system~\eqref{eq:RSS} with all subsystems stable, assume that for some positive constants $\{\tau_p\}_{p\in\mathcal{P}}$, there exists a collection of symmetric, positive definite matrices $\{Q_p\in\,M_n(\mathbb{R})\colon\ p\in\mathcal{P}\}$ such that \begin{eqnarray*}
								Q_pA_p+A_p^\top Q_p&<&0\,\,\,\,\,\text{ for all }\, p\in\mathcal{P},\\
								{\rm{e}}^{A_q^\top\tau_q}R_{(q,p)}^\top Q_p R_{(q,p)}{\rm{e}}^{A_q\tau_q}&\le& Q_q\,\text{ for all }\,p,q\in\mathcal{P},
							\end{eqnarray*}
							then the reset switched system~\eqref{eq:RSS} is stable for all signals $\sigma\in\mathcal{S}_s^M[\{\tau_p\}_{p\in\mathcal{P}}]$.
						\end{theorem}
						
						\begin{theorem}[Generalization of \Cref{thm:MLF:issue5}]
							Given the impulsive switched system~\eqref{eq:ISS} with all subsystems stable with a compact set of impulses $\mathcal{I}$, assume that for some positive constants $\{\tau_p\}_{p\in\mathcal{P}}$, there exists a collection of symmetric, positive definite matrices $\{Q_p\in\,M_n(\mathbb{R})\colon\ p\in\mathcal{P}\}$  such that \begin{eqnarray*}
								Q_pA_p+A_p^\top Q_p&<&0,\,\,\,\,\,\text{ for all }\, p\in\mathcal{P},\\
								{\rm{e}}^{A_q^\top\tau_q}M^\top Q_p M{\rm{e}}^{A_q\tau_q}&\le& Q_q,\,\text{ for all }\,p,q\in\mathcal{P}\text{ and }M\in\mathcal{I},
							\end{eqnarray*}
							are satisfied. Then the impulsive switched system~\eqref{eq:ISS} is stable under arbitrary impulses for all signals $\sigma\in\mathcal{S}_s^M[\{\tau_p\}_{p\in\mathcal{P}}]$.
						\end{theorem}
						
						We now extend the result obtained for the case when both stable and unstable subsystems are present to obtain mode-dependent dwell times and flee times constraints for stability of reset switched system. We define the following sets:
						\begin{eqnarray*}
							N_s^p(0,t)=\#\{i\colon\, 0\le t_i< t,\,\sigma(t_i)=p\in\mathfrak{s}\} &,&
							N_u^p(0,t)=\#\{i\colon\, 0\le t_i< t,\,\sigma(t_i)=p\in\mathfrak{u}\},\\
							N(0,t)&=&\sum_{p\in\mathfrak{s}}N_s^p(0,t)+\sum_{p\in\mathfrak{u}}N_u^p(0,t).
						\end{eqnarray*}
						
						\begin{theorem}[Generalization of \Cref{thm:MLF:mixed}]
							For a given the reset switched system~\eqref{eq:RSS}, suppose there exist a collection of symmetric, positive definite matrices $\{Q_p\in\,M_n(\mathbb{R})\colon\ p\in\mathcal{P}\}$, positive constants $\lambda_p$, for $p\in \mathfrak{s}$, $\mu_p$, for $p\in\mathfrak{u}$, $\gamma_{(q,p)}$ ($\gamma_{(q,p)}<1$ when $\mathfrak{s}=\emptyset$) satisfying \begin{eqnarray*}
								Q_pA_p+A_p^\top Q_p\le -\lambda_p Q_p, \text{ for all }\,p\in\mathfrak{s} &,& Q_pA_p+A_p^\top Q_p\le\mu_p Q_p, \text{ for all }\,p\in\mathfrak{u},\\
								R_{(q,p)}^\top Q_pR_{(q,p)}&\le& \gamma_{(q,p)}Q_q, \text{ for all }\,p,q\in\mathcal{P}.
							\end{eqnarray*}
							Let $\gamma=\max_{p,q\in\mathcal{P}}\,\{\gamma_{(q,p)}\}$. Then the reset switched system~\eqref{eq:RSS} is stable for any signal $\sigma\in\nobreak\mathcal{S}^{M}\left(\{\tau_p\}_{p\in\mathfrak{s}},\{\eta_p\}_{p\in\mathfrak{u}}\right)$ satisfying \[\lim\sup_{t\to\infty}\left(-\sum_{p\in\mathfrak{s}}\lambda_p\, N_s^p(0,t)\,\tau_p+\sum_{p\in\mathfrak{u}}\mu_p\, N_u^p(0,t) \,\eta_p+N(0,t)\,\ln\gamma\right)<0.\]
						\end{theorem}
						
						\begin{theorem}[Generalization of \Cref{thm:MLF:mixed:arbitrary}]
							Given an impulsive switched system~\eqref{eq:ISS} with a compact set of impulses $\mathcal{I}$. Let there exist a collection of symmetric, positive definite matrices $\{Q_p\in\,M_n(\mathbb{R})\colon\ p\in\mathcal{P}\}$, constants $\lambda_p>0$ for $p\in \mathfrak{s}$, $\mu_p>0$ for $p\in\mathfrak{u}$, and $\gamma_{(p,q)}\ge 1$ for $p,q\in\mathcal{P}$ satisfying \begin{eqnarray*}
								Q_pA_p+A_p^\top Q_p\le -\lambda_p Q_p, \text{for all }\,p\in\mathfrak{s} &,& Q_pA_p+A_p^\top Q_p\le \mu_p Q_p, \text{for all }\,p\in\mathfrak{u},\\
								M^\top Q_pM&\le& \gamma_{(q,p)}Q_q, \text{for all }\,p,q\in\mathcal{P}\text{ and }M\in\mathcal{I}.
							\end{eqnarray*}
							Let $\lambda=\min_{p\in\mathfrak{s}}\,\{\lambda_p\}$, $\mu=\max_{p\in\mathfrak{u}}\,\{\mu_p\}$, and $\gamma=\max_{p,q\in\mathcal{P}}\,\{\gamma_{(p,q)}\}$. Then the impulsive switched system~\eqref{eq:ISS} is stable under arbitrary impulses for any signal $\sigma\in\nobreak\mathcal{S}^{M}\left(\{\tau_p\}_{p\in\mathfrak{s}},\{\eta_p\}_{p\in\mathfrak{u}}\right)$ satisfying \[\lim\sup_{t\to\infty}\left(-\sum_{p\in\mathfrak{s}}\lambda_p\, N_s^p(0,t)\,\tau_p+\sum_{p\in\mathfrak{u}}\mu_p\, N_u^p(0,t) \,\eta_p+N(0,t)\,\ln\gamma\right)<0.\]
						\end{theorem}

						\section{Examples}\label{sec:examples}
						
						In this section, we present some examples illustrating our results. We will use complex Jordan basis matrix having unit-norm columns for each subsystem unless otherwise specified. In the following examples, we will use the following subsystem matrix $A_i$ with the indicated Jordan decomposition $P_iJ_iP_i^{-1}$.
						
						\begin{table}[h!]
							\centering
							\begin{tabular}{|c|c|c|c|}
								\hline 
								$i$ & $A_i$& $P_i$ & $J_i$ \\ \hline\hline
								1,2& $\begin{psmallmatrix}
									-0.1&2\\-1&-0.1
								\end{psmallmatrix}$ & $\frac{1}{\sqrt{3}}\begin{psmallmatrix}
									i\,\sqrt{2}&-i\,\sqrt{2}\\1&1
								\end{psmallmatrix}$ &$\begin{psmallmatrix}
									-0.1-i\,\sqrt{2}&0\\ 0& -0.1+i\,\sqrt{2}
								\end{psmallmatrix}$	\\ \hline
								3 &$\begin{psmallmatrix}
									-0.1&-1\\2&-0.1
								\end{psmallmatrix}$ & $\sqrt{\frac{2}{3}}\begin{psmallmatrix}
									-i\,/\sqrt{2}&i\,/\sqrt{2}\\1&1
								\end{psmallmatrix}$& $\begin{psmallmatrix}
									-0.1-i\,\sqrt{2}&0\\ 0& -0.1+i\,\sqrt{2}
								\end{psmallmatrix}$\\ \hline
								4 &  $\begin{psmallmatrix}
									-5&-3&-4\\4&2&4\\0&0&-1
								\end{psmallmatrix}$ & $\begin{psmallmatrix}
									-1/\sqrt{2}&1/\sqrt{2}&2/\sqrt{21}\\1/\sqrt{2}&0&-4/\sqrt{21}\\0&-1/\sqrt{2}&1/\sqrt{21}
								\end{psmallmatrix}$ & $\begin{psmallmatrix}
									-2&0&0\\0&-1&0\\0&0&-1
								\end{psmallmatrix}$ \\ \hline
								5 & $\begin{psmallmatrix}
									0&2&1\\ -2&1&0\\1&-2&0
								\end{psmallmatrix}$ & $\begin{psmallmatrix}
									0&-1/\sqrt{3}&-1/\sqrt{3}\\
									-1/\sqrt{5}&\frac{1}{2\sqrt{3}}\left(-1+i\,\sqrt{3}\right)&\frac{1}{2\sqrt{3}}\left(-1-i\,\sqrt{3}\right)\\
									2/\sqrt{5}&1/\sqrt{3}&1/\sqrt{3}
								\end{psmallmatrix}$ & $\begin{psmallmatrix}
									1&0&0\\0&-i\, \sqrt{3}&0\\0&0&i\, \sqrt{3}
								\end{psmallmatrix}$  \\ \hline
								6 & $\begin{psmallmatrix}
									-2&1&0\\1&-2&0\\0&0&-3
								\end{psmallmatrix}$ & $\begin{psmallmatrix}
									0&-1/\sqrt{2}&1/\sqrt{2}\\0&1/\sqrt{2}&1/\sqrt{2}\\1&0&0
								\end{psmallmatrix}$ & $\begin{psmallmatrix}
									-3&0&0\\0&-3&0\\0&0&-1
								\end{psmallmatrix}$ \\ \hline
								7 & $\begin{psmallmatrix}
									-1&1&1\\0 &-2&2\\0&0&-3
								\end{psmallmatrix}$ &  $\begin{psmallmatrix}
									1/\sqrt{21}&-1/\sqrt{2}&1\\-4/\sqrt{21}&1/\sqrt{2}&0\\2/\sqrt{21}&0&0
								\end{psmallmatrix}$ & $\begin{psmallmatrix}
									-3&0&0\\0&-2&0\\0&0&-1
								\end{psmallmatrix}$ \\
								\hline
							\end{tabular}
							\caption{Subsystem matrices}
						\end{table}
												
						\begin{example}[Stable systems with destabilizing impulses]\label{eg:destabiss}
							This example is a well-known example from Liberzon~\cite{liberzon2003switching} for state-dependent switching, which can also be viewed as a time-dependent switched system. We present it as a reset switched system~\eqref{eq:RSS}. Consider the system $\dot{x}=A_1x$. Suppose the system is reset periodically at time instances $t_k=2^{-3/2}k\pi$, $k\in\mathbb{N}$ given by $x(t_k) =\begin{psmallmatrix}
								0&1\\-1&0
							\end{psmallmatrix} x(t_k^-)$. The linear reset map, despite satisfying $\left\|\begin{psmallmatrix}
								0&1\\-1&0
							\end{psmallmatrix}\right\|=1$, destabilizes the system. This is true since the discrete-time system $x(t_k)=\begin{psmallmatrix}
								0&1\\-1&0
							\end{psmallmatrix}{\rm{e}}^{A_1\frac{\pi}{2\sqrt{2}}}\, x(t_{k-1})$ is unstable as the spectral radius $\rho\left(\begin{psmallmatrix}
								0&1\\-1&0
							\end{psmallmatrix}{\rm{e}}^{A_1\frac{\pi}{2\sqrt{2}}}\right)=1.26>1$, refer~\cite[Corollary 1.1]{jungers}.
							However we know that if the resets occur not too fast, then the stability can be ensured. To obtain the dwell time guaranteeing stability of the reset switched system, we use \Cref{thm:dtRSS} having subsystems $A_1=A_2$ and reset maps $R_{(1,2)}=R_{(2,1)}=\begin{psmallmatrix}
								0&1\\-1&0
							\end{psmallmatrix}$. The dwell time constraint for stability using \Cref{thm:DTforRSS_flow} is $\tau=3.47$. The Robust Control Toolbox in MATLAB 2021a returns the dwell time constraint for stability of the reset switched system as $\tau=3.4032$. This means that the reset system defined above is stable when the time difference between any two resets is at least 3.4032 time units. 
					\end{example}

						\begin{example}[Dwell time and mode-dependent dwell time constraints]
							\label{eg:bplssbistab}
							Consider a switched system~\eqref{eq:system} with two subsystem matrices $A_2$ and $A_3$. By \Cref{cor:StabilizingResets}, the reset switched system having subsystems $A_2$ and $A_3$ can be stabilized under arbitrary switching by using resets $R_{(2,3)}=\text{diag}\left(-1/\sqrt{2}, \sqrt{2}\right)$ and $R_{(3,2)}=\text{diag}\left(-\sqrt{2}, 1/\sqrt{2}\right)$ at the switching instances. This means that the corresponding reset switched system~\eqref{eq:RSS} is stable under arbitrary switching. Further for the system under consideration with $R_{(2,3)}=\begin{psmallmatrix}
								2&3\\1&2
							\end{psmallmatrix}$, $R_{(3,2)}=\begin{psmallmatrix}
								1&-2\\-2&5
							\end{psmallmatrix}$, the corresponding reset switched system is stable for any signal with dwell time $\tau=17.1$ using \Cref{thm:dtRSS}, computed using the Robust Control Toolbox in MATLAB R2021a. The dwell time obtained using \Cref{thm:DTforRSS_flow} is $\tau_R=20.34$ (with Jordan decompositions given above). The mode-dependent dwell time for stability obtained using \Cref{thm:gen:DTforRSS} yields $\tau_2=20.34$ and $\tau_3=14.96$. As discussed before, using mode-dependent dwell time constraints allows us to obtain a larger class of stabilizing signals than the dwell time constraint.
						\end{example}
						
						\begin{example}[Stability under arbitrary impulses]\label{eg:3dArbReset}
							Consider the switched system~\eqref{eq:system} with three subsystems $A_8=\begin{psmallmatrix}
								-5&3&-3\\0&-2&2\\0&0&-1
							\end{psmallmatrix}$, $A_9=\begin{psmallmatrix}
								-2&2&-1\\4&3&-4\\7&10&-10
							\end{psmallmatrix}$ and $A_{10}=\begin{psmallmatrix}
								-1&-2&-3\\1&0&1\\0&-1&-3
							\end{psmallmatrix}$ with the set of impulses $\mathcal{I}=\text{conv}\{M_1,M_2,M_3\}$, where $M_1=\begin{psmallmatrix}
								-2&1&0\\0&2&-1\\3&0&0
							\end{psmallmatrix}$, $M_2=\begin{psmallmatrix}
								-3&2&-1\\1&4&2\\-2&-1&1
							\end{psmallmatrix}$ and $M_3=\begin{psmallmatrix}
								1&1&-1\\2&0&2\\1&0&3
							\end{psmallmatrix}$. Then the conditions~\eqref{eq:mlf1cpt} and~\eqref{eq:mlf2cpt} in \Cref{thm:MLF:issue5} reduce to a set of 21 LMIs as discussed in \textit{point 1} of \Cref{sec:convexhull}. The impulsive switched system~\eqref{eq:ISS} is, thus, stable under arbitrary impulses for any signal with dwell time $\tau=2.81$. Using \Cref{thm:ISS:flow} and \textit{point 3} in \Cref{sec:convexhull}, the dwell time for stability under arbitrary impulses is obtained as $\tau_A=3.48$.
						\end{example}

						\begin{example}[Dwell-flee relation]\label{eg:RSS:mixed}
							Consider a reset switched system with two subsystems $A_4$ and $A_5$, and reset maps $R_{(4,5)}=\begin{psmallmatrix}
								3&5&-2\\-6&-4&1\\3&1&2
							\end{psmallmatrix}$, $R_{(5,4)}=\begin{psmallmatrix}
								-7&3&0\\0&-7&-3\\-5&0&4
							\end{psmallmatrix}$. 
							The condition~\eqref{eq:acyclic_cond} with $\epsilon=1$ is not satisfied by Jordan basis matrices $P_4$ and $P_5$ of $A_4$ and $A_5$, respectively. However, upon appropriate scaling as per \Cref{lemma:acyclic_cond}, the Jordan basis matrices $P_4$ and $(10^{-3}) P_5$ of $A_4$ and $A_5$, respectively, satisfy condition~\eqref{eq:acyclic_cond} with $\epsilon=1$ and we obtain $\tau_R=6.96$, $\eta_R=2.33$ using \Cref{thm:DTforRSS_flow}. 
							
							
							Moreover, the Robust Control Toolbox shows that the conditions of \Cref{thm:MLF:mixed} are feasible for $\lambda=\lambda_1=1$, $\mu=\mu_2=2$ and $\gamma=75$. This implies that the reset switched system is stable for any signal $\mathcal{S}[\tau,\eta]$ satisfying\\ $\lim\sup_{t\to\infty}\left(-N_s(0,t)\,\tau+2\, N_u(0,t) \,\eta+N(0,t)\,\ln 75\right)<0$, which can be reduced to $\tau>2\eta+2\ln 75$. This reduction is possible since $	\lim_{t\to\infty}\,N_s(0,t)/N_u(0,t)=1$, and hence $\lim_{t\to\infty}\,N(0,t)/N_u(0,t)=\lim_{t\to\infty}\,N(0,t)/N_s(0,t)=1$. As a consequence, the switched system is stable under a periodic switching signal $\sigma_2$ which spends exactly $\tau=14.64$ time in stable subsystem and exactly $\eta=3$ in unstable subsystem.
							
						\end{example}
						
						\begin{example}[Scope for improvements]\label{eg:scope} Through this example, we establish that the results obtained in this paper can be further improved, that is, lower dwell time constraints can be obtained. The dwell time calculations in this section have been done by taking complex Jordan basis matrix $P_i$ of $A_i$ having unit-norm columns, and by analyzing the flow in the Euclidean norm. The unit-norm column condition is just used for the sake of uniformity. Any choice of Jordan basis matrices for the subsystems would yield different bounds as discussed earlier. Consider a reset switched system with two subsystem matrices $A_6$ and $A_7$ and reset matrices $R_{(6,7)}=\begin{psmallmatrix}
								0&1&0\\0&0&1\\1&0&0
							\end{psmallmatrix}$ and $R_{(7,6)}=\begin{psmallmatrix}
								-1&0&1\\0&1&0\\0&0&-1
							\end{psmallmatrix}$. Using \Cref{thm:DTforRSS_flow}, we obtain the dwell time for stability as $\tau_R=1.44$. Note that 
							\begin{itemize}
								\item if the unit-norm column condition is dropped and, instead, we use Jordan basis matrices $V_6=\begin{psmallmatrix}
									0&-1&1\\0&1&1\\1&0&0
								\end{psmallmatrix}$ and $V_7=\begin{psmallmatrix}
									1&-1&1\\-4&1&0\\2&0&0
								\end{psmallmatrix}$ of $A_6$ and $A_7$, respectively, we obtain the dwell time as $\max\left\{\ln \|V_7^{-1}R_{(6,7)}V_6\|,\,\ln \|V_6^{-1}R_{(7,6)}V_7\|\right\}=1.38$.
								\item if the underlying norm is the ellipsoidal norm $\|x\|_A=\sqrt{x^\top A x}$, where $A$ is a real positive definite symmetric matrix, then it induces an operator norm on $M_n(\mathbb{R})$. Moreover, for any $K\in M_n(\mathbb{R})$, $\|K\|_A=\sqrt{\lambda_{max}\left(A^{-1}K^\top A K\right)}$, refer~\cite{phdthesis}. Using this characterization, a lower dwell time constraint can be obtained when the underlying norm is $\|\cdot\|_A$ where $A=\text{diag}\left(2,1/2,1\right)$. We obtain $\tau_R^2=\max\left\{\ln \|V_7^{-1}R_{(6,7)}V_6\|_A,\,\ln \|V_6^{-1}R_{(7,6)}V_7\|_A\right\}=1.3$. 
							\end{itemize}
						\end{example}

						\section{Concluding remarks}
						
						In this paper, we consider reset (impulsive) switched systems. Using the flow of the switched system and multiple Lyapunov theory, we obtain time constraints on signals which stabilize these switched systems. The results obtained in this paper are simplified when there is a convex hull of reset/impulse matrices. The results are also extended to obtain mode dependent dwell time conditions for stability of reset (impulsive) switched systems. Extensions of these results to average dwell time and mode dependent average dwell time constraints is an ongoing project. While using the flow to obtain time constraints for stability, it is possible to obtain lower dwell time on considering other norms on the state space, as exhibited in \Cref{eg:scope}. This issue can be explored further. Furthermore, stability of systems on introduction of more general impulses can be investigated.
						
						\section{Declarations}
						
						\subsection*{Conflicts of Interest}

							Authors declare that they have no conflict of interest.
							
						\subsection*{Data Availability Statement}
						
							No data is used to support this study.

\bibliographystyle{plain}
\bibliography{References}
\end{document}